\documentclass[a4paper,notitlepage,twoside,reqno,12pt]{amsart}

\usepackage{mathrsfs}
\usepackage{bbm,pifont,latexsym}
\usepackage{anysize}  
\marginsize{3cm}{3cm}{3cm}{3cm}

\usepackage{dcolumn,indentfirst,color,hyperref}
\usepackage{amsmath,amssymb,amscd,amsthm,amsfonts,mathrsfs}
\usepackage{color,graphicx,xcolor,graphics}
\usepackage{titlesec} 
\usepackage{titletoc} 
\titlecontents{section}[20pt]{\addvspace{2pt}\filright}
              {\contentspush{\thecontentslabel. }}
              {}{\titlerule*[8pt]{}\contentspage}

\usepackage{fancyhdr} 
\newtheorem{thm}{Theorem}[section]

\newtheorem{theorem}{Theorem}

\newtheorem{lem}[thm]{Lemma}
\newtheorem{cor}[thm]{Corollary}
\newtheorem{prop}[thm]{Proposition}
\newtheorem{pro}[thm]{Proposition}

\newcommand{\C}{\mathbb{C}}

\makeatletter\@addtoreset{equation}{section}\makeatother 

\titleformat{\section}{\centering\normalsize}{\textsc{\thesection.}}{1em}{\textsc}
\titleformat{\subsection}{\normalsize}{\thesubsection.}{1em}{\textbf}

\begin{document}

\title{Quasisymmetric rigidity of Sierpi\'nski carpets  $F_{n,p}$}

\author{Jinsong Zeng}
\address{Jinsong Zeng, School of Mathematics, Fudan University, 200433, Shanghai, P. R. China}
\email{10110180006@fudan.edu.cn}

\author{Weixu Su}
\address{Weixu Su, School of Mathematics, Fudan University, 200433, Shanghai, P. R. China}
\email{suweixu@gmail.com}

\begin{abstract}
We study a new class of square Sierpi\'nski carpets $F_{n,p}$ ($5\leq n, 1\leq p<\frac{n}{2}-1$) on $\mathbb{S}^2$,
which are not quasisymmetrically equivalent to the standard Sierpi\'{n}ski
carpets. We prove that the group of quasisymmetric self-maps of each $F_{n,p}$ is the Euclidean isometry group.
 We also establish that  $F_{n,p}$ and $F_{n',p'}$ are quasisymmetrically equivalent if and only if $(n,p)=(n',p')$.
\end{abstract}



\date{\today}

\thanks{J. Zeng is partially supported by NSFC 11201078; W. Su is partially supported by NSFC 11201078, CPSF 2012M510778 and Shanghai Postdoctoral Sustentation Fund (Grant No. 12R21410400).}

\maketitle
\section{Introduction}\label{sec_introduction}
The quasisymmetric geometry of Sierpi\'nski carpets is related to the
study of Julia sets in complex dynamics and boundaries of Gromov hyperbolic groups.
For background and research progress,
we recommend the survey of M. Bonk \cite{Bo2}.

Let $\mathbb{S}^2$ be the unit sphere in $\mathbb{R}^3$.
Let $S=\mathbb{S}^2\setminus \bigcup_{i\in \mathbb{N}} D_i$ be the complement in $\mathbb{S}^2$ of countably many
pair-wise disjoint open Jordan regions $D_i \subset \mathbb{S}^2$.
$S$ is called a \emph{(Sierpi\'nski) carpet} if $S$ has empty interior, diam $(D_i)\to 0$ as $i\to \infty$,
and $\partial D_i\cap \partial D_j=\emptyset$ for all $i\neq j$.
The boundary of $D_i$, denoted by $C_i$, is called a \emph{peripheral circle} of $S$.
A \emph{round carpet } is a carpet on  $\mathbb{S}^2$ such that all of its peripheral circles
are geometric circles. Typical Examples of round carpets are limit sets of convex co-compact Kleinian groups.

Topologically all carpets are the same
\cite{Why}. Much richer structure arises if we consider quasisymmetric geometry of metric carpets.
The famous conjecture of Kapovich-Kleiner \cite{KK} predicts that if $G$ is a hyperbolic group
with  boundary  $\partial_\infty G$
homeomorphic to a Sierpi\'nski carpet, then $G$ acts geometrically (the action is isometrical, properly discontinuous and co-compact) on a convex
subset of $\mathbb{H}^3$ with non-empty totally geodesic boundary.
The Kapovich-Kleiner conjecture is equivalent to the conjecture that the carpet $\partial_\infty G$ (endowed with
the ``visual'' metric) is
quasisymetriclly equivalent to a round carpet on  $\mathbb{S}^2$.
The conjecture is true for carpets that can be quasisymmetrically embedding in $\mathbb{S}^2$ \cite{Bo1}.

The concept of quasisymmetric map between metric spaces was
 defined by Tukia and V\"ais\"al\"a \cite{TV}. Let $f: X\to Y$ be a homeomorphism between two metric spaces $(X,d_X)$ and $(Y,d_Y)$.
$f$ is \emph{quasisymmetric} if there exists a homeomorphism $\eta: [0,\infty)\to [0,\infty)$ such that
$$\frac{d_Y(f(x),f(y))}{d_Y(f(x),f(z))}\leq \eta (\frac{d_X(x,y)}{d_X(x,z)}), \  \forall \ x,y,z\in X, x\neq z.$$
It follows from the definition that the quasisymmetric self-maps of $X$ form a group $\mathrm{QS}(X)$.

A homeomorphism $f: X\to Y$ is called \emph{quasi-M\"obius} if there exists a homeomorphism $\eta: [0,\infty)\to [0,\infty)$ such that
for all 4-tuple $(x_1,x_2,x_3,x_4)$ of distinct points in $X$, we have
$$[f(x_1),f(x_2),f(x_3),f(x_4)]\leq \eta ([x_1,x_2,x_3,x_4]),$$
where
$$[x_1,x_2,x_3,x_4]=\frac{d_X(x_1,x_3)d_X(x_2,x_4)}{d_X(x_1,x_4)d_X(x_2,x_3)}$$
is the \emph{metric cross-ratio}.

It is not hard  to check that a quasisymmetric map between metric spaces is  quasi-M\"obius. Conversely, any
quasi-M\"obius map between bounded metric spaces is quasisymmetric \cite{TV}.

An important tool in the study of quasisymmetric maps is the conformal modulus
of a given family of paths. The notion of conformal modulus (or extremal length)
was first introduced by Beurling and Ahlfors \cite{BA}. It has many applications in complex analysis
and metric geometry \cite{LV,Hei}.
In the work of Bonk and Merenkov \cite{BM},
it was proved that every quasisymmetric self-homeomorphism
of the standard $1/3$-Sierpi\'nski carpet $S_3$ is a Euclidean
isometry. For the standard $1/p$-Sierpi\'nski carpets $S_p$, $p \geq 3 $ odd, they showed that the groups $\mathrm{QS}(S_p)$ of
 quasisymmetric self-maps are finite dihedral. They also established
that $S_p$ and $S_q$ are quasisymmetrically equivalent if only if $p=q$.
The main tool in their proof
is  the  \emph{carpet modulus}, which is a certain discrete modulus of a path family and
is preserved under quasisymmetric maps of carpets.

\begin{figure}[!hbp]
\centering
\includegraphics[width=0.45\linewidth]{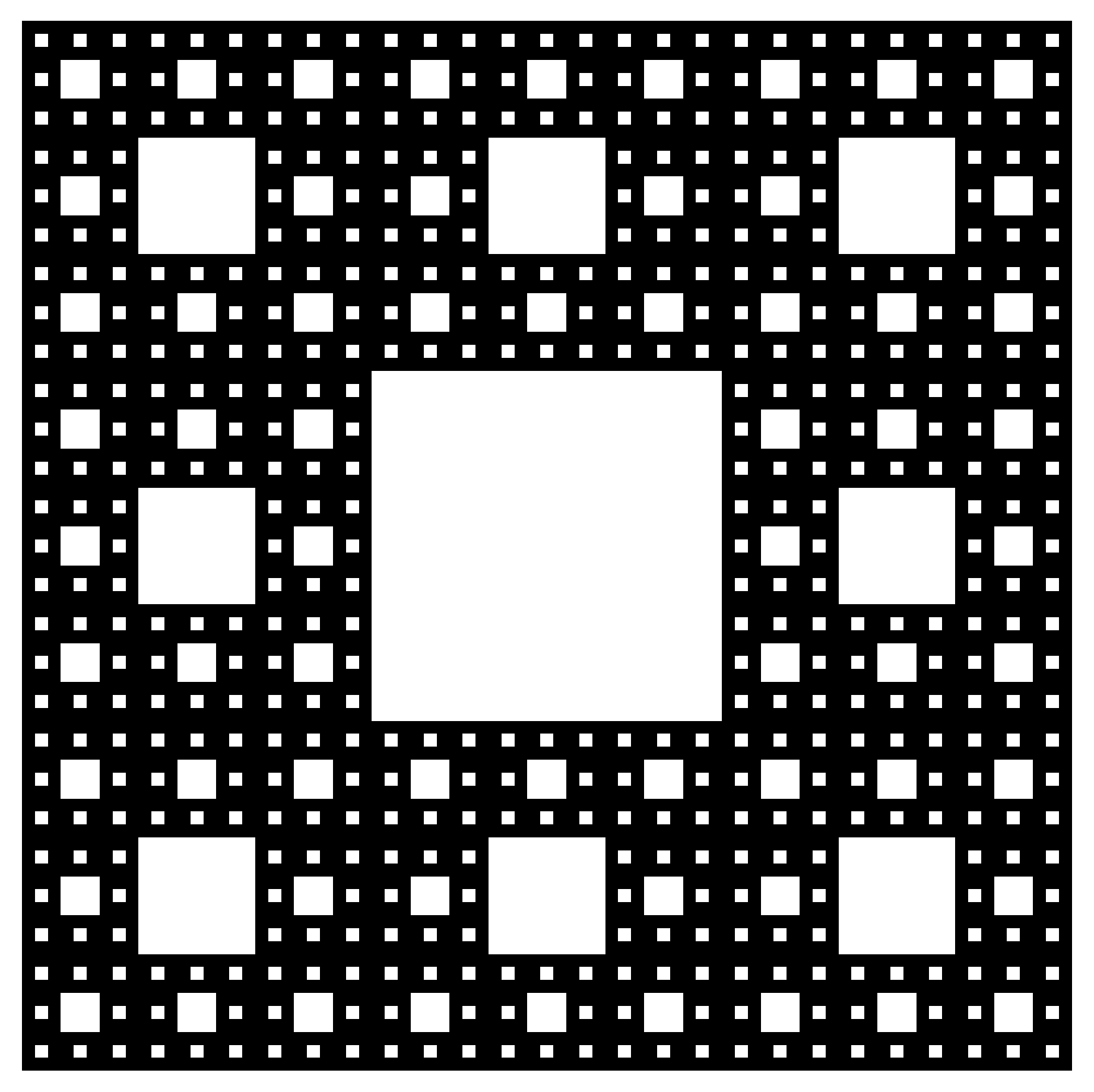}
\caption{The standard Sieipi\'nski carpet $S_3$.}
\label{fig:carpet}
\end{figure}

The aim of this paper is to extend Bonk-Merenkov's results to a new class of Sierpi\'nski carpets.
Unless otherwise indicated, we will equip a carpet $S=\mathbb{S}^2\setminus \bigcup_{i\in \mathbb{N}} D_i$  with the spherical
metric. Note that when a carpet is contained in a compact set $K$ of $\mathbb{C}\subset \mathbb{C}\cup \{\infty\}\cong \mathbb{S}^2$,
the Euclidean and the spherical metrics are bi-Lipschitz equivalent on $K$.

\subsection{Main results}  Let $5\leq n, 1\leq p<\frac{n}{2}-1$ be integers. Let $Q^{(0)}_{n,p}=[0,1]\times[0,1]$ be
the closed unit square in $\mathbb{R}^2$.
We first subdivide $Q^{(0)}_{n,p}$ into $n^2$ subsquares with equal side-length $1/n$ and remove the interior of four subsquares, each
 has side-length $1/n$ and is of
 distance $\sqrt{2}p/n$ to  one of the four corner points of $Q_{n,p}^{(0)}$.

 The resulting set $Q^{(1)}_{n,p}$ consists of $(n^2-4)$ squares of side-length $1/n$. Inductively, $Q^{(k+1)}_{n,p}$,
$k\geq 1$, is obtained from $Q_{n,p}^{(k)}$ by subdividing each of the remaining squares in the subdivision of
$Q^{(k)}_{n,p}$ into $n^2$ subsquares of equal side-length $1/n^{k+1}$ and removing the interior of four subsquares as we have
done above.

The Spierpi\'nski carpet $F_{n,p}$ is the intersection of all the sets $Q^{(k)}_{n,p}$, i.e.,
$$F_{n,p}=\bigcap_{k=0}^{+\infty}Q_{n,p}^{(k)}.$$ See Figure \ref{fig:carpet}.

\begin{figure}[!hbp]
\centering
\includegraphics[width=0.45\linewidth]{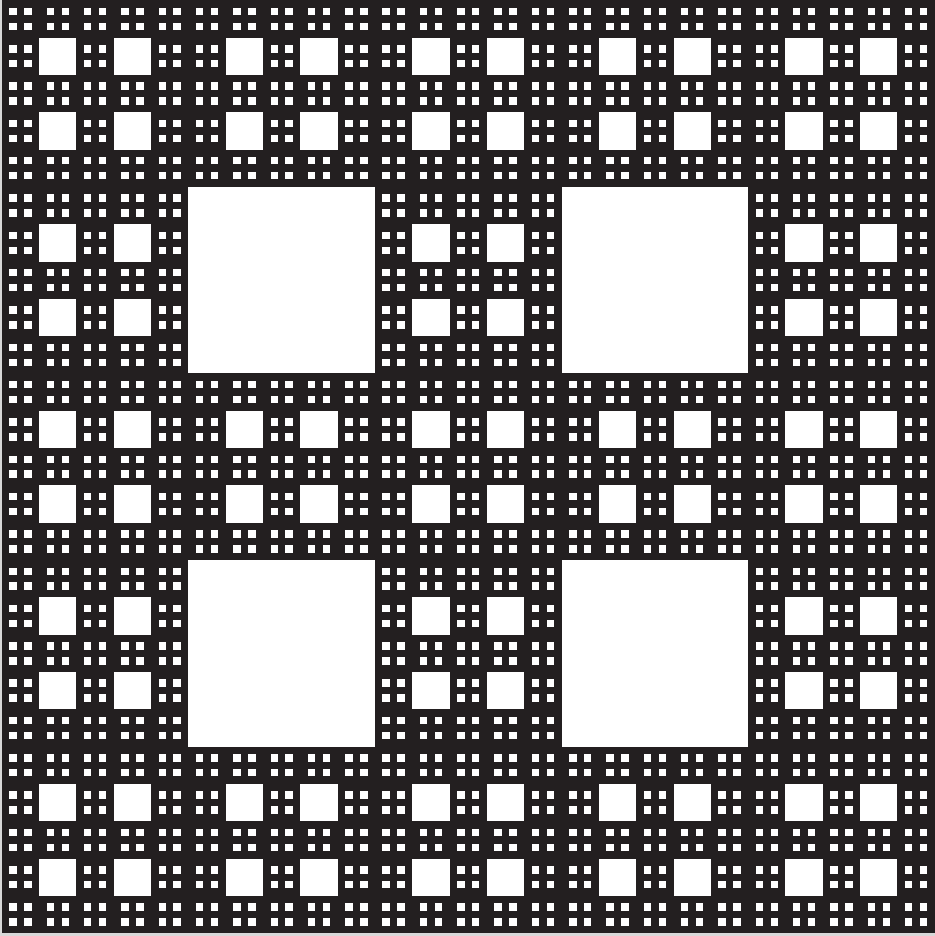}
\caption{The carpet $F_{5,1}$.}
\label{fig:carpet}
\end{figure}

The following theorem will be proved in Section \ref{sec_distinguished_pairs_of_peripheral_circles}.
It shows that, from the point of view of quasiconformal geometry,
 the carpets $F_{n,p}$ are different with the standard Sierpi\'nski carpets $S_m, m\geq 3$ odd
 (note that the standard Sierpi\'nski carpets $S_m$ is constructed from a similar process, by removing the interior of
 the middle square in each steps).

\begin{theorem}\label{thm:sp}
Let $5\leq n, 1\leq p<\frac{n}{2}-1$ be integers.  The carpet $F_{n,p}$ is not quasisymmetrically equivalent to
the Standard Sierpi\'nski carpet $S_m, m\geq 3$ odd.
\end{theorem}

It was proved by Bonk and Merenkov \cite{BM} that for $m \geq 3 $ odd the quasisymmetric group
$\mathrm{QS}(S_m)$ is a finite dihedral group. Moreover, when $m=3$, $\mathrm{QS}(S_3)$
is the Euclidean isometry group of $S_3$. In Section \ref{sec_proof_of_theorems}, we will show that
\begin{theorem}\label{thm:rigidity}
Let $f$ be a quasisymmetric self-map of $F_{n,p}$. Then $f$ is a Euclidean isometry.
\end{theorem}
Note that the Euclidean isometric group  of $F_{n,p}$ (and $S_m$), consists of eight elements, is
the group generated by the reflections in the diagonal $\{(x,y)\in \mathbb{R}^2 : x=y\}$
and the vertical line $\{(x,y)\in \mathbb{R}^2 : x=\frac{1}{2}\}$.

We will also prove that
\begin{theorem}\label{thm:equivalent}
Two  Sierpi\'{n}ski carpets $F_{n,p}$ and $F_{n',p'}$ are quasisymmetrically equivalent if and only if $(n,p)=(n',p')$.
\end{theorem}
\subsection{Idea of the proofs}
The main tools to prove the above theorems are the \emph{carpet modulus} and the \emph{weak tangent}, both
of which were
investigated in \cite{BM}. Our arguments follow the same outline as \cite{BM}.

We will first concentrate on carpet modulus of the families of curves
connecting the boundary of the annulus domains bounded by pairs of distinct peripheral circles of $F_{n,p}$.
The extremal mass distribution of such a carpet modulus exists and is unique (Proposition \ref{extremal_mass}).
This, together with the auxiliary results in Section \ref{sec:auxiliary}, allows us to show that (see Section 4)
any quasisymmetric self-map $f$ of $F_{n,p}$ should preserves the set $\{O, M_1, M_2,M_3,M_4\}$,
where $O$ is the boundary of the unit square and $M_1, M_2,M_3,M_4$ are the boundary of the first four squares
removed from the unit square.

It is more difficult to see that $f$ should maps $O$ to $O$.
To show this, we first study the weak tangents of the carpets (this is our main work on Section \ref{sec_weak_tangent_spaces}).
In Section \ref{sec_proof_of_theorems}, we prove that $f(O)=O$ by counting the orbit of a corner of $O$
or $M_i$ under the group $\mathrm{QS}(F_{n,p})$.

\subsection{Remark}
Our arguments in this paper apply to a  more general class of Sierpi\'nski Carpets $F_{n,p,r},r\geq 1,p\geq1,n\geq5,1\leq p+r<\frac{n}{2}$.
Let $Q_{n,p,r}^{(0)}=[0,1]\times[0,1]$.
Subdivide $Q_{n,p,r}^{(0)}$ into $n^2$ subsquares and remove the interior of four bigger subsquares with side-length $r/n$
and is of distance $\sqrt{2}p/n$ to one of the four corner points of $Q_{n,p,r}^{(0)}$. So the resulting set $Q_{n,p,r}^{(1)}$
has $(n^2-4r^2)$ subsquares with side-length $1/n$. Repeating the operation to the subsquares, we obtain $Q_{n,p,r}^{(2)}$.
Inductively, we have $Q_{n,p,r}^{(k)}$.

Then the carpet $F_{n,p,r}=\bigcap_{k\geq0}Q_{n,p,r}^{(k)}$.
See Figure \ref{fig:F_{7,1,2}}.
Note that $F_{n,p}=F_{n,p,1}$.

Similarly, $F_{n,p,r}$ is not quasisymmetrically equivalent to $S_m, m\geq3$ odd and $QS(F_{n,p,r})$ is the isometric group.
Moreover, $F_{n,p,r}$ and $F_{n',r',p'}$ are quasisymmetrically equivalent if and only if $(n,p,r)=(n',p',r')$.
Since the proof of the above conclusions are of no essential difference from that of $F_{n,p}$, we shall omit it.

\begin{figure}[!hbp]
\centering
\includegraphics[width=0.45\linewidth]{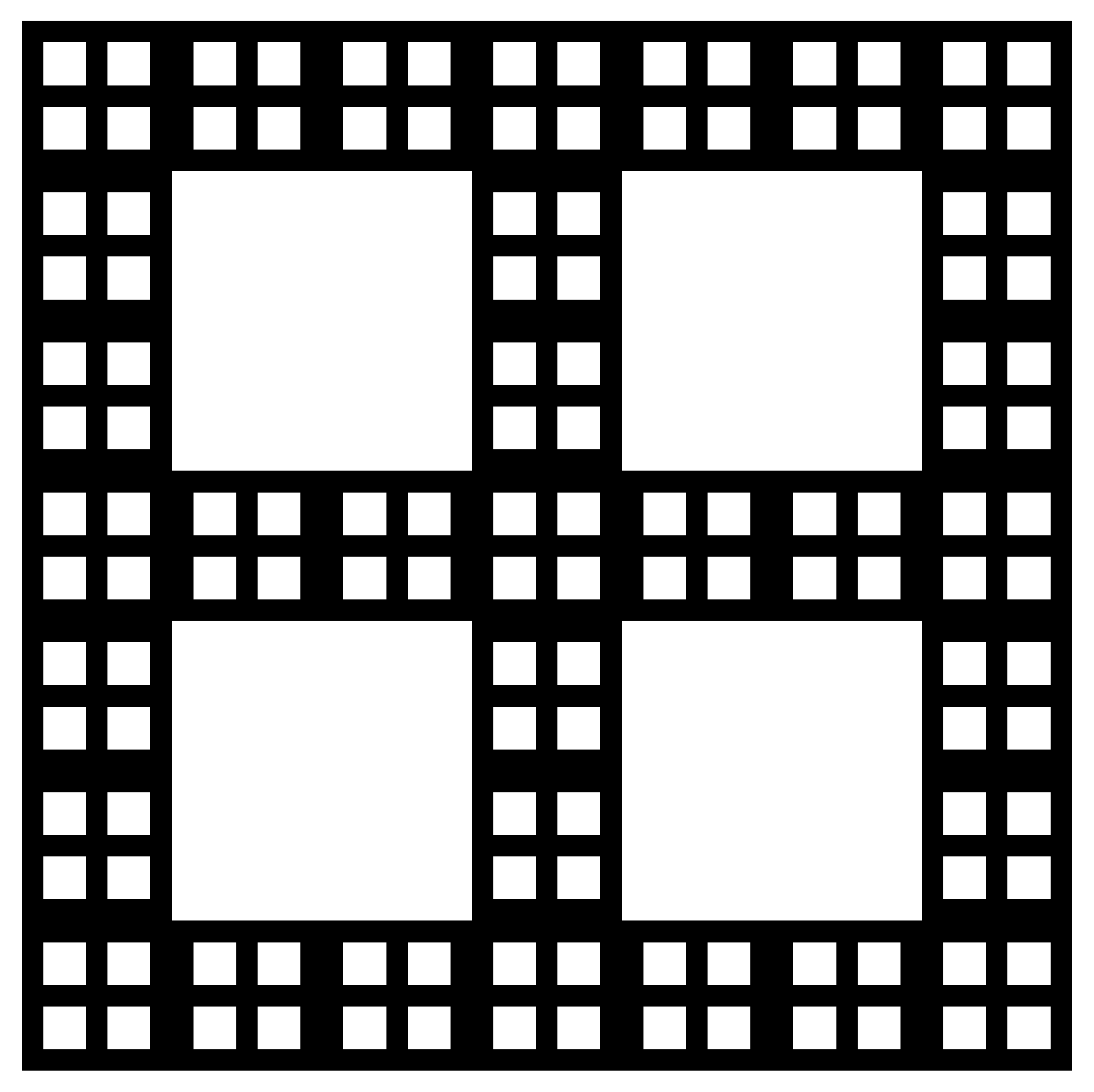}
\caption{The carpet $F_{7,1,2}$.}
\label{fig:F_{7,1,2}}
\end{figure}
\section{Carpet modulus}\label{sec:carpet_modulus}

In this section, we shall recall the definitions of conformal modulus and carpet modulus. The carpet modulus was
 introduced by Bonk-Merenkov \cite{BM} as a quasisymmetric invariant.
 There are several important properties of the carpet modulus that will be used in the rest of
our paper. In many cases, we will neglect the proof and refer to  \cite{BM} instead.

\subsection{Conformal modulus} A \emph{path} $\gamma$ in a metric space $X$ is a continuous map $\gamma: I\rightarrow X$ of a finite interval $I$.
Without cause of confusion, we shall identified the map with its image
$\gamma(I)$ and denote a path by $\gamma$.
We say that $\gamma$ is \emph{open} if $I=(a, b)$. The limits $\lim_{t\rightarrow a}\gamma(t)$ and
$\lim_{t\rightarrow b}\gamma(t)$, if they exist, are called the \emph{end points of $\gamma$}. If $A, B\subseteq X$, then we say that \emph{$\gamma$
connects $A$ and $B$} if $\gamma$ has endpoints such that one of them lies in $A$ and the other lies in $B$. If $I=[a,b]$ is a closed interval, then the
length of $\gamma: I\rightarrow X$ is defined by
\[\textup{length}(\gamma):=\textup{sup}\sum_{i=1}^{n}|\gamma(t_{i})-\gamma(t_{i-1})|\]
where the supremum is taken over all finite sequences $a=t_{0}\leq t_{1}\leq t_{2}\leq \cdots \leq t_{n}=b$. If $I$ is not closed, then
we set $$\textup{length}(\gamma):=\sup_J {\textup{length}(\gamma|J)}, $$
where $J$ is taken over all closed subintervals
of $I$ and $\gamma|_J$ denotes the restriction of $\gamma$ on $J$. We call $\gamma$ \emph{rectifiable} if its length is finite.
Similarly, a path $\gamma: I\rightarrow X$ is \emph{locally rectifiable} if its restriction to each closed subinterval is rectifiable.
Any rectifiable path $\gamma: I\rightarrow X$ has a unique extension $\overline{\gamma}$ to the closure $\overline{I}$ of $I$.

Let $\Gamma$ be a family of paths in $\mathbb{S}^2$. Let $\sigma$ be the spherical measure and $ds$ be the spherical line element
on $\mathbb{S}^2$ induced by the spherical metric (the Riemannian metric on $\mathbb{S}^2$ of constant curvature $1$).
The \emph{conformal modulus} of $\Gamma$ is defined as
$$\text{mod}(\Gamma):=\text{inf}\int_{\mathbb{S}^2}\rho^2d\sigma\ ,$$
where the infimum is taken over all nonnegative Borel functions $\rho:\mathbb{S}^2\rightarrow[0,\infty]$ satisfying
\begin{equation*}\label{equ:admissible}
\int_\gamma\rho ds\geq1
\end{equation*}
for all locally rectifiable path $\gamma\in\Gamma$. Functions $\rho$ satisfying ($\ref{equ:admissible}$) for all locally rectifiable path $\gamma\in\Gamma$ are called \emph{admissible}.

It is easy to show that (see \cite{Ahl})

\begin{equation}\label{equ:modulus1}
\text{mod}(\Gamma_1)\leq\text{mod}(\Gamma_2),
\end{equation}
if $\Gamma_1\subseteq\Gamma_2$ and
\begin{equation}\label{equ:modulus2}
\text{mod}(\bigcup_{i=1}^\infty\Gamma_i)\leq\sum_{i=1}^{\infty}\text{mod}(\Gamma_i).
\end{equation}
Moreover, if $\Gamma_1$ and
$\Gamma_2$ are two families of paths such that each path $\gamma$ in $\Gamma_1$ contains a subpath $\gamma'\in \Gamma_2$, then
\begin{equation}\label{equ:modulus3}
\text{mod}(\Gamma_1)\leq\text{mod}(\Gamma_2)
\end{equation}

If $f:\Omega\rightarrow\Omega'$ is a continuous map between domains $\Omega$ and $\Omega'$ in $\mathbb{S}^2$ and $\Gamma$ is a
family of paths contained in $\Omega$, then we denote by $f(\Gamma)=\{f\circ\gamma \ | \ \gamma\in\Omega\}$.

If $f:\Omega\rightarrow\Omega'$ is a conformal map between regions
$\Omega$, $\Omega'\subseteq\mathbb{S}^2$ and $\Gamma$ is a family of paths in $\Omega$, then mod($\Gamma$)=mod($f(\Gamma)$).
This is the fundamental property of modulus: conformal maps do not change the conformal modulus of a family of paths.

In this paper,
we shall adopt the
metric definition of quasiconformal maps (\cite{HK}, Definition 1.2) and allow them to be orientation-reversing.
Suppose that $f: X\to Y$ is a homeomorphism between two metric spaces $X$ and $Y$.
$f$ is \emph{quasiconformal} if there is a constant $H\geq 1$, s.t. $\forall x\in X$,
$$\lim\sup_{r\to 0^+}\frac{\max\{d(f(x),f(y)):d(x,y)\leq r\}}{\min\{d(f(x),f(y)):d(x,y)\geq r\}}\leq H.$$

Quasiconformal maps distort the conformal modulus of path families in a controlled way.
Let $\Omega$ and $\Omega'$ be regions in $\mathbb{S}^2$ and let $\Gamma$ be a family of paths in $\Omega$. Suppose that
 $f:\Omega\rightarrow\Omega'$
is quasiconformal map. Then
\begin{equation}\label{equ:qc}
\frac{1}{K}\textup{mod}(\Gamma)\leq\textup{mod}(f(\Gamma))\leq K\textup{mod}(\Gamma),
\end{equation}
where $K\geq 1$ depends on the dilatation of $f$ .

From
 $(\ref{equ:qc})$, a quasiconformal map preserves the modulus
of a path family up to a fixed multiplicative constant. So if $\Gamma_0\subseteq\Gamma$ and $\textup{mod}(\Gamma_
0)=0$, then $\textup{mod}(f(\Gamma_0))=0$.

\subsection{Carpet modulus}
If a certain property for paths in $\Gamma$ holds for all paths outside an exceptional family $\Gamma_0\subseteq\Gamma$
with $\mathrm{mod}(\Gamma_0)=0$, then we say that it holds for \emph{almost every path in $\Gamma$}.

Let $S=\mathbb{S}^2\backslash\bigcup_{i=1}^{\infty}D_i$ be a carpet with $C_i=\partial D_i$, and let $\Gamma$ be a family
of paths in $\mathbb{S}^2$.
A \emph{mass distribution} $\rho$ is a function that assigns to each $C_i$ a non-negative number $\rho(C_i)$.

The \emph{carpet modulus} of $\Gamma$ with respect to $S$ is defined as
$$\textup{mod}_S(\Gamma)=\mathop{\textup{inf}}^{}_{\rho}\textup{ }\sum_{i}\rho(C_i)^2,$$
where the infimum is taken over all
\emph{admissible} mass distribution $\rho$, that is, mass distribution $\rho$ satisfies
\begin{equation*}\label{admissible_condition}
\mathop{\sum}^{}_{\gamma\bigcap C_i\neq\emptyset}\rho(C_i)\geq1
\end{equation*}
for all most every path in $\Gamma$.

It is straightforward to check that the carpet modulus is momotone and countably subadditive, the same properties as conformal modulus in
(\ref{equ:modulus1}), (\ref{equ:modulus2}) and (\ref{equ:modulus3}). An crucial property of carpet modulus
is its invariance under quasiconformal maps.

\begin{lem}[\cite{BM}]\label{lem:invariant}
Let $D, \widetilde{D}\subset \mathbb{S}^2$ be regions and $f:D\to D$ be a quasiconformal map.
Let $S\subseteq D$ be a carpet and $\Gamma$ be a family of paths such that $\gamma\subset D$ for each
$\gamma\in \Gamma$.  Then $$\textup{mod}_{f(S)}(f(\Gamma))=\textup{mod}_S(\Gamma).$$
\end{lem}

\subsection{Carpet modulus with respect to a group}
We also need the notion of \emph{carpet modulus with respect to a group}.

Let $S=\mathbb{S}^2\setminus\bigcup_{i\in\mathbb{N}}D_i$ be a carpet and $C_i=\partial D_i$.
Let $G$ be a group of homeomorphisms of $S$. If $g\in G$ and $C\subseteq S$ is a peripheral circle of $S$,
 then $g(C)$ is also a peripheral circle of $S$. Let $\mathcal{O}=\{g(C) : g\in G\}$  be the orbit of $C$ under
 the action of $G$.

Let $\Gamma$ be a familly of paths in $\mathbb{S}^2$.
A admissible $G$-invariant mass distribution $\rho: \{C_i\}\to [0,+\infty]$ is a mass distribution such that

\begin{enumerate}
\item $\rho(g(C))=\rho(C)$ for all $g\in G$ and all peripheral circles $C$ of $S$;
\item  almost every path $\gamma$ in $\Gamma$  satisfies
$$\sum_{\gamma\bigcap C_i\neq\emptyset}\rho(C_i)\geq 1.$$
\end{enumerate}
The \emph{carpet modulus} $\textup{mod}_{S/G}(\Gamma)$ with respect to the action of $G$ on $S$ is defined as $$\textup{mod}_{S/G}(\Gamma):=\mathop{\textup{inf}}^{}_{\rho}\sum_{\mathcal{O}}\rho(\mathcal{O})^2,$$ where the infimum is taken over all admissible $G$-invariant mass distributions.
In the above definition, $\rho(\mathcal{O})$ is defined by $\rho(C)$ for any
$C\in \mathcal{O}$. Since $\rho$ is $G$-invariant, $\rho(\mathcal{O})$ is well-defined.
Note that each orbit contributions with exactly one term to the sum $\sum_{\mathcal{O}}\rho(\mathcal{O})^2$.

\begin{lem}[\cite{BM}]\label{lem:group carpet}
Let $D$ be a region in $\mathbb{S}^2$ and $S$ be a carpet contained in $D$. Let $G$ be a group of homeomorphisms on $S$. Suppose that
 $\Gamma$ is
a family of paths with $\gamma\subseteq D$ for each $\gamma\in\Gamma$ and $f:D\rightarrow\widetilde{D}$ a quasiconformal map onto
another region $\widetilde{D}\subseteq\mathbb{S}^2$. We denote $\widetilde{S}=f(S)$, $\widetilde{\Gamma}=f(\Gamma)$ and
$\widetilde{G}=(f|_S)\circ G\circ(f|_S)^{-1}$, then $$\textup{mod}_{\widetilde{S}/\widetilde{G}}(\widetilde{\Gamma})
=\textup{mod}_{S/G}(\Gamma).$$
\end{lem}

\begin{lem}\label{cyclic}
Let $S$ be a carpet in $\mathbb{S}^2$ and $\Psi: \mathbb{S}^2\rightarrow\mathbb{S}^2$ be a quasiconformal map with $\Psi(S)=S$, $\psi:=\Psi|_S$. Assume that $\Gamma$ is a $\Psi$-invariant path family in $\mathbb{S}^2$ such that for every peripheral circle $C$ of $S$ that meets some path in $\Gamma$ we have $\psi^n(C)\neq C$ for all $n\in\mathbb{Z}$. Then $\textup{mod}_{S/\langle\psi^k\rangle}(\Gamma)=k\textup{mod}_{S/\langle\psi\rangle}(\Gamma)$ for every $k\in\mathbb{N}$.
\end{lem}
This is (\cite{BM}, Lemma 3.3). In this Lemma, $\langle\psi\rangle$ denotes the cyclic group of homeomorphisms on $S$ generated by $\psi$, and $\Gamma$ is called \emph{$\Psi$-invariant} if $\Psi(\Gamma)=\Gamma$. This lemma gives a precise relationship between the carpet modulus with respect to a cyclic group and its subgroups.

\subsection{Existence of extremal mass distribution}
Let $S=\mathbb{S}^2\setminus \{D_i\}, C_i=\partial D_i$ be a carpet and $\Gamma$ be a family of paths on $\mathbb{S}^2$.
An admissible mass distribution $\rho$ for a carpet modulus $\mathrm{mod}_S(\Gamma)$ is called
\emph{extremal} if $\mathrm{mod}_S(\Gamma)$ is obtained by $\rho$:
$$\mathrm{mass}(\rho)=\sum_{i}\rho(C_i)^2=\mathrm{mod}_S(\Gamma).$$
Similarly, an $G$-invariant mass distribution that obtains $\textup{mod}_{S/G}(\Gamma)$ is also called \emph{extremal}.

A criterion for the existence of an extremal mass distribution for carpet modulus (with respect to the group)
is given by \cite{BM}. Recall that the peripheral circles $\{C_i\}$ are \emph{uniform quasicircles}
if there exists a homeomorphism $\eta:[0,\infty)\to [0,\infty)$ such that every $C_i$ is the image
of an $\eta$-quasisymmetric map of the unit circle.

\begin{prop}\label{extremal}
Let $S$ be a carpet in $\mathbb{S}^2$ whose peripheral circles are uniform quasicircles, and let $\Gamma$ be an arbitrary path family in $\mathbb{S}^2$ with $\textup{mod}_S(\Gamma)<+\infty$. Then the extremal mass distribution for $\textup{mod}_S(\Gamma)$ exists and is unique.
\end{prop}
This is (\cite{BM}, Proposition 2.4). The uniqueness follows from elementary convexity argument.

\begin{pro}\label{extremal-group}
Let $S$ be a carpet in $\mathbb{S}^2$ whose peripheral circles are uniform quasicircles. Let $G$ be a group of homeomorphisms of $S$
 and $\Gamma$ be a path family in $\mathbb{S}^2$ with $\textup{mod}_{S/G}(\Gamma)<+\infty$. Suppose that for each $k\in \mathbb{N}$ there exists a family of peripheral circles $\mathcal{C}_k$ of $S$ and a constant $N_k\in\mathbb{N}$ with the following properties:
\begin{enumerate}
\item If $\mathcal{O}$ is any orbit of peripheral circles of $S$ under the action of $G$, then
$\#(\mathcal{O}\bigcap\mathcal{C}_k)\leq N_k$ for all $k\in \mathbb{N}.$
\item If $\Gamma_k$ is the family of all paths in $\Gamma$ that only meet peripheral circles in $\mathcal{C}_k$, then $\Gamma=\bigcup_k\Gamma_k.$
\end{enumerate}
Then extremal mass distribution for $\textup{mod}_{S/G}(\Gamma)$ exists and is unique.
\end{pro}
This is (\cite{BM}, Proposition 3.2).

\section{Auxiliary results}\label{sec:auxiliary}
In this section, we collect a series of results obtained by M. Bonk and his coauthors \cite{BKM,Bo1}.
The theorems and propositions cited here are the cornerstone of our later proof
(as well as they were for the proof in \cite{BM}).

\subsection{Quasiconformal extention of quasisymmetric map}
\begin{pro}\label{quasiconformal_extend}
Let $S$ be a carpet in $\mathbb{S}^2$  whose peripheral circles are uniform quasicircles and let $f$ a quasisymmetric map of $S$ onto another carpet $\widetilde{S}\subseteq\mathbb{S}^2$. Then there exists a self-quasiconformal map $F$ on $\mathbb{S}^2$ whose restriction to $S$ is $f$.
\end{pro}
This is (\cite{Bo1}, Proposition 5.1).

\subsection{Quasisymmetric uniformization and rigidity}
The peripheral circles $\{C_i \}$ of $S$  are called \emph{uniformly relatively separated} if the pairwise distances are uniformly bounded away from zero. i.e., there exists $\delta>0$ such that $$\Delta(C_i,C_j)=\frac{\textup{dist}(C_i,C_j)}{\textup{min}\{\textup{diam}(C_i),\textup{diam}(C_j)\}}\geq\delta$$ for any two distinct $i$ and $j$. This property is preserved under quasisymmetric maps. See (\cite{Bo1}, Corollary 4.6).

\begin{thm}\label{uniformization_by_round_carpets}
Let $S$ be a carpet in $\mathbb{S}^2$ whose peripheral circles are uniformly relatively separated uniformly quasicircles, then there exists
a quasisymmetric map of $S$ onto a round carpet.
\end{thm}
This is (\cite{Bo1}, Corollary 1.2). Recall that a carpet $S=\mathbb{S}^2\setminus\bigcup D_i$ is called \emph{round} if each $D_i$ is an open spherical disk.

\begin{thm}\label{quasisymmetric_rigidity_of_round_carpets}
Let $S$ be a round carpet in $\mathbb{S}^2$ of measure zero. Then every quasisymmetric map of $S$ onto any other round carpet is the restriction
of a M\"{o}bius transformation.
\end{thm}

This is (\cite{BKM}, Theorem 1.2).
Here by definition a M\"{o}bius transformation is a fractional linear transformation on $\mathbb{S}^2\cong\hat{\mathbb{C}}$ or the complex-conjugate of such a map. So we allow a M\"{o}bius transformation to be orientation-reversing.
\subsection{Three-Circle Theorem}
Let $S\subseteq\mathbb{S}^2$ be a carpet. A homeomorphism embedding $f: S\rightarrow\mathbb{S}^2$
 is called \emph{orientation-preserving} if  some homeomorphic extension $F:\mathbb{S}^2\rightarrow\mathbb{S}^2$ of
 $f$ is orientation-preserving on $\mathbb{S}^2$ (such an extension exists and the definition is independent of
 the choice of extension, see the proof of Lemma 5.3 in \cite{Bo1}).

\begin{cor}\label{three_circle_theorem}
Let $S$ be a carpet in $\mathbb{S}^2$ of measure zero whose peripheral circles are uniformly relatively separated uniform quasicircles and  $C_i, i=1,2,3$ be three distinct peripheral circles of $S$. Let $f$ and $g$ be two orientation-preserving quasisymmetric self-maps of $S$.
Then we have the following rigidity results:
\begin{enumerate}
\item Assume that $f(C_i)=g(C_i)$ for $i=1,2,3$. Then $f=g$.

\item Assume that $f(C_i)=g(C_i)$ for $i=1,2$ and $f(p)=g(p)$ for a given point $p\in S$. Then $f=g$.

\item  Assume that $G$ is the group of all orientation-preserving quasisymmetric self-maps of $S$ that fix $C_1$, $C_2$. Then $G$ is a finite cyclic group.

\item Assume that $G$ is the group of all orientation-preserving quasisymmetric self-maps of $S$ that fix $C_1$ and fix a given point $q\in C_1$, then $G$ is an infinite cyclic group.
\end{enumerate}
\end{cor}
\begin{proof}
The proof we given here is contained in \cite{BM}. Since its conclusion is important for
the rest of our paper, we include it here for completeness.

By Theorem \ref{uniformization_by_round_carpets}, there exists a quasisymmetric map $h$ of $S$ onto a round carpet $\widetilde{S}$. Using Proposition \ref{quasiconformal_extend}, we can extend $h$ to a quasiconformal map on $\mathbb{S}^2$. Since quasiconformal maps preserve the class of sets of measure zero, $\widetilde{S}$ has measure zero as well. We denote by $G_0$ and $\widetilde{G_0}$ the group of all orientation-preserving quasisymmetric self-maps of $S$ and $\widetilde{S}$, respectively. By the quasisymmetric rigidity of round carpets (Theorem \ref{quasisymmetric_rigidity_of_round_carpets}), $\widetilde{G_0}$ consists of the restriction of orientation-preserving M\"{o}bius transformations that fix $\widetilde{S}$.

Now we look at the homomorphism $h_*$ induced by $h$:
\begin{eqnarray*}\label{isomorphic}
h_*: G_0 &\rightarrow &\widetilde{G_0}, \\
  \psi &\mapsto& h\circ\psi\circ h^{-1}.
\end{eqnarray*}
We can check that $h_*$ is well-defined and is an isomorphic. Since $h_*(f)$ and $h_*(g)$
are orientation-preserving M\"{o}bius transformation and $h_*(f)\circ(h_*(g))^{-1}$ fixes distinct spherical round circles $h(C_i)$, $i=1,2,3$,
we know that $h_*(f)\circ(h_*(g))^{-1}=\textup{id}$ and $(1)$ follows.

We can prove $(2)$ from the fact that any orientation-preserving M\"{o}bius transformation
 fixing distinct spherical round circles and a given non-common center point $p\in \mathbb{S}^2$ is the identity.

To prove $(3)$,  it suffices to show that $\widetilde{G}=h_*(G)$ is a finite cyclic. By post-composing fractional linear transformation to $h$, we can assume that $h(C_1)$ and $h(C_2)$ are distinct spherical round circles with the same center. Note that $\widetilde{G}$ consists of orientation-preserving M\"{o}bius transformation, fixing $h(C_1)$, $h(C_2)$ and $\widetilde{S}$. Moreover, $\widetilde{G}$ must be a discrete group as it maps peripheral round circles of $S$ to peripheral round circles. Hence $\widetilde{G}$ is a finite cyclic group, then $(3)$ follows.

For $(4)$, similarly, by post-composing fractional linear transformation to $h$, we can assume that $h(C_1)=\mathbb{R}\bigcup\{\infty\}$, $h(q)=0$ and $\widetilde{S}$ is contained in the upper half-plane. Then the maps in $\widetilde{G}$ are of the form: $z\mapsto\lambda z$ with $\lambda>0$, fixing $\widetilde{S}$. By the same reason as $(3)$, $\widetilde{G}$ is a discrete group. So there exists a $\lambda_0\geq1$ such that $\widetilde{G}=\{z\mapsto\lambda_0^n z|n\in\mathbb{N}\}$. It follows that $\widetilde{G}$, and hence also $G$, is the trivial group consisting only of the identity or an infinite cyclic group. Therefore, $(4)$ follows.
\end{proof}

\subsection{Square carpets}
A $\mathbb{C}^*$-$Cylinder$ $A$ is a set of the form $$A=\{z\in\mathbb{C};r\leq|z|\leq R\}$$ with $0<r<R<+\infty$. The metric on $A$ induced by the length element $|dz|/|z|$ which is the flat metric. Equipped with this metric, $A$ is isometric to a finite cylinder of circumference $2\pi$ and length $\textup{log}(R/r)$. The boundary components $\{ z\in \mathbb{C};|z|=r\}$ and $\{z\in\mathbb{C};|z|=R\}$ are called the inner and outer boundary components of $A$, respectively.

A $\mathbb{C}^*$-$square$ $Q$ is a Jordan region of the form $$Q=\{\rho e^{i\theta}:a<\rho<b,\alpha<\theta<\beta\}$$ with $0<\textup{log}(b/a)=\beta-\alpha<2\pi$. We call the quantity $$l_{\mathbb{C}^*}(Q)=\textup{log}(b/a)=\beta-\alpha$$ its side length. Clearly, two opposite sides of $Q$ parallel to the boundaries of $A$, while the other two perpendicular to the boundaries of $A$.

A square carpet $T$ in a $\mathbb{C}^*$-cylinder $A$ is a carpet that can be written as $$T=A\setminus\bigcup_i Q_i,$$ where the sets $Q_i$, $i\in I$, are $\mathbb{C}^*$-squares whose closures are pairwise disjoint and contained in the interior of $A$.

\begin{thm}
Let $S$ be a carpet of measure zero in $\mathbb{S}^2$ whose peripheral circles are uniformly relatively separated uniform quasicircles, and $C_1$ and $C_2$ two distinct peripheral circles of $S$. Then there exists a quasisymmetric map $f$ from $S$ onto a square carpet $T$ in a $\mathbb{C}^*$-cylinder $A$ such that $f(C_1)$ is the inner boundary component of $A$ and $f(C_2)$ is the outer one.
\end{thm}
This is (\cite{Bo1}, Theorem 1.6).

Let $S$ be a carpet in $\mathbb{S}^2$ and $C_1, C_2$ be two distinct peripheral circles
of $S$. Soppose that the Jordan regions $D_1$ and $D_2$ are the complementary components of
$S$ bounded by $C_1$ and $C_2$ respectively.
We let $\Gamma(C_1, C_2)$ be the family of all open paths
in $S^2\setminus \overline{D}_1\cup \overline{D}_2$ that connects $\overline{D}_1$ and $\overline{D}_2$.

\begin{pro}\label{extremal_mass}
Let $S$ be a carpet of measure zero in $\mathbb{S}^2$ whose peropheral circles are uniformly relatively separated uniform quasicircles, and $C_1$ and $C_2$ two distinct peripheral circles of $S$. Then

$(1)$ $\textup{mod}_S(\Gamma(C_1,C_2))$ has finite and positive total mass.

$(2)$ Let $f$ be a quasisymmetric map of $S$ onto a square carpet $T$ in a $\mathbb{C}^*$-cylinder $A=\{z\in\mathbb{C};r\leq |z|\leq R\}$ such that $C_1$ corresponds to the inner and $C_2$ to the outer boundary components of $A$. Then the extremal mass distribution is given as follows: $$\rho(C_1)=\rho(C_2)=0, \  \rho(C)=\frac{l_{\mathbb{C}^*}(f(C))}{\textup{log}(R/r)}$$ with the peripheral circles $C\neq C_1, C_2$ of $S$.
\end{pro}

This is (\cite{Bo1}, Corollary 12.2).

Let $S$ be a carpet in a closed Jordan region $D\subset\hat{\mathbb{C}}$. $S$ is called $square\ carpet$ if
$\partial D$ is a peripheral circle of $S$ and all other peripheral circles are squares with sides parallel to the coordinate
axes.
\begin{thm}\label{thm_square_rigidity}
Let $S$ and $\widetilde{S}$ be square carpets of measure zero in rectangles $K=[0,a]\times[0,1]\subseteq\mathbb{R}^2$ and $\widetilde{K}=[0,\widetilde{a}]\times[0,1]\subseteq\mathbb{R}^2$, respectively, where $a$, $\widetilde{a}>0$. If $f$ is an orientation-preserving quasisymmetric homeomorphism form $S$ onto $\widetilde{S}$ that takes the corners of $K$ to the corners of $\widetilde{K}$ with $f(0)=0$. Then $a=\widetilde{a}$, $S=\widetilde{S}$, and $f$ is the identity on $S$.
\end{thm}

This is (\cite{BM}, Theorem 1.4). Here the expression square carpet $S$ in a rectangle $K$ means that a carpet $S\subset K$ so that $\partial K$ is a peripheral circle of $S$ and all other peripheral circles are squares with four sides parallel to the sides of $K$, respectively.

\section{Distinguished  peripheral circles}\label{sec_distinguished_pairs_of_peripheral_circles}
Let $n\geq 5, 1\leq p<\frac{n}{2}-1$ be integers.
 Let $F_{n,p}$ be the Sierpi\'nski carpet as we defined in the introduction.
 We endow $F_{n,p}$ with the Euclidean metric in $\mathbb{R}^2$. Since $F_{n,p}$ is
 a subset of $[0,1]\times [0,1]$, the Euclidean metric (measure) is comparable with the spherical metric (measure).

If $Q$ is a peripheral circle of $F_{n,p}$, we denote by $\ell_{Q}$ the Euclidean side length of $Q$. Denote by
$Q_0$ the unit square $[0,1]\times [0,1]$.

\begin{lem}\label{F_pq_carpet_measure_zero_quasicircles_separated}
The carpet $F_{n,p}$ is of measure zero. The peripheral circles of $F_{n,p}$ are uniform quasicircles and uniformly relatively separated.
\end{lem}

\begin{proof}
It follows from the construction that $F_{n,p}$ is a carpet of Hausdorff dimension
$$\log (n^2-4)/\log n<2.$$ So the measure of $F_{n,p}$ is equal to zero.

Since each peripheral circle of $F_{n,p}$ can be mapped to the boundary of
$Q_0$ by a Euclidean similarity,
 the peripheral circles of $F_{n,p}$ are uniform quasicircles.

At last, the peripheral circles of $F_{n,p}$ are uniformly relatively separated in the Euclidean metric. Indeed, consider
any two distinct peripheral circles  $C_1, C_2$ of $F_{n,p}$. The Euclidean distance between $C_1$ and $C_2$
satisfies
\begin{eqnarray*}
\textup{dist}(C_1,C_2)&\geq& \min\{\ell(C_1), \ell(C_2)\} \\
&=& \frac{1}{\sqrt{2}}\min\{\mathrm{diam}(C_1), \mathrm{diam} (C_2)\}.
\end{eqnarray*}
\end{proof}

\subsection{Distinguished pairs of non-adjacent peripheral circles}
We denote by $O$ the boundary of the unit square $Q_0$.
In the first step of the inductive construction of $F_{n,p}$, there are four squares $Q_1, Q_2,Q_3,Q_4$ of side-length $\frac{1}{n}$, i.e., the lower left, lower right, upper right and upper left squares respectively, removed from $Q_0$. We denote by $M_i, i=1, \cdots, 4$ the boundary of $Q_i, i=1, \cdots, 4$, respectively.

In the following discussions, we call $O$ the \emph{ outer circle} of $F_{n,p}$ and $M_i, i=1, \cdots, 4$ the
\emph{inner circles} of $F_{n,p}$.
We say that two disjoint peripheral circles $C,C'$ are \emph{adjacent} if there exists
a copy $F$ of $F_{n,p}$ (here $F\subset F_{n,p}$ can be considered as a carpet scaled from $F_{n,p}$ by some factor $1/n^k$)
such that $C,C'$ are inner circles of $F$.
For example, two distinct inner circles $M_i$ and $M_j$ are adjacent.
 Two disjoint peripheral circles $C,C'$ which are not adjacent are called
\emph{non-adjacent}.

\begin{lem}\label{interchange}
Let $\{C,C'\}$ be any pair of non-adjacent distinct peripheral circles of $F_{n,p}$. Then
$$\mathrm{\mod}_{F_{n,p}}(\Gamma(C,C'))\leq \mathrm{\mod}_{F_{n,p}}(\Gamma(O, M)).$$
 Moreover, the equality holds if and only if $\{C,C'\}=\{O,M\}$ for some inner circle $M=M_i$.
\end{lem}
\begin{proof}
Assume that $\{C,C'\}\neq \{O, M\}$ for any inner circle $M$.
By Lemma $\ref{F_pq_carpet_measure_zero_quasicircles_separated}$ and Proposition ($\ref{extremal_mass}$), $\textup{mod}_{F_{p,q}}(\Gamma
(C,C'))$ is a finite and positive number.
Without loss of generality we may assume that $\ell(C)=1/n^m\leq \ell(C')$.
Note that there exists a copy $F\subset F_{n,p}$, rescaled from $F_{n,p}$ by a factor $1/n^{m-1}$,
so that $C$ corresponds to some inner circle, say, $M_1$ of $F_{n,p}$.

Denote the outer circle of $F$ by $C_0$. Since $C$ and $C'$ are disjoint and $\ell(C)\leq \ell(C')$,
$C'$ is disjoint with the interior region of $C_0$.
Hence every path in $\Gamma(C,C')$ must intersect with $C_0$ and
then contains a sub-path in $\Gamma(C,C_0)$.
See Figure \ref{fig:carpet} for an illustration.
\begin{figure}[!hbp]
\centering
\includegraphics[width=0.78\linewidth]{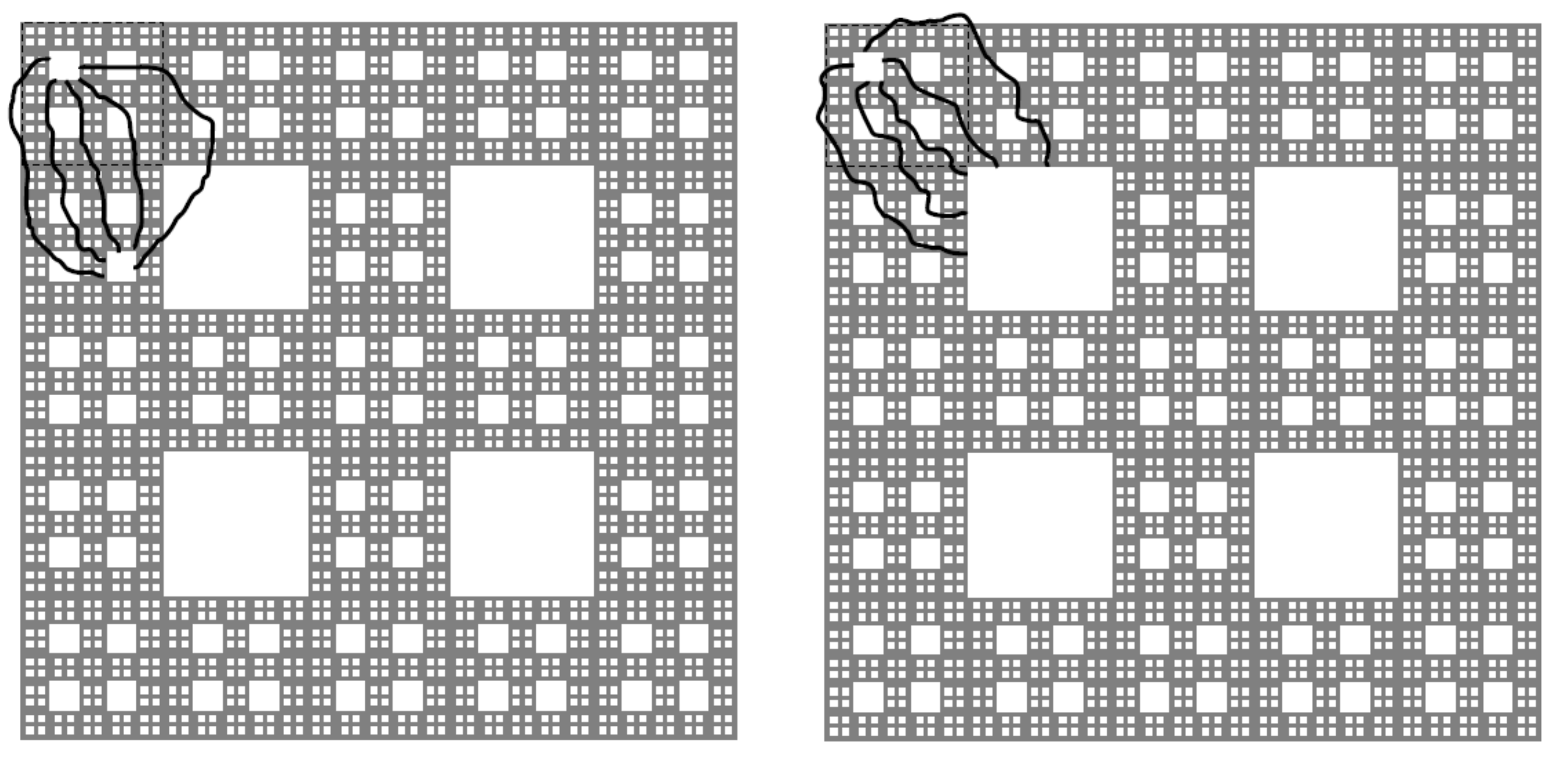}
\caption{The non-adjacent distinct peripheral circles}
\label{fig:carpet}
\end{figure}

 Therefore
\begin{equation}\label{similarity}
\textup{mod}_{F_{n,p}}(\Gamma(C,C'))\leq\textup{mod}_{F_{n,p}}(\Gamma(C,C_0)).
\end{equation}
On the other hand, since every path in $\Gamma(C,C_0)$ meets exactly the same peripheral circles
of $F$ and $F_{n,p}$,
we have  $$\textup{mod}_{F_{n,p}}(\Gamma(C,C_0))=\textup{mod}_{F}(\Gamma(C,C_0)).$$
Moreover, by the similarity of $F_{n,p}$ and $F$,
$$\textup{mod}_{F}(\Gamma(C,C_0))=\textup{mod}_{F_{n,p}}(\Gamma(M,O)).$$
It follows that
$$\mathrm{\mod}_{F_{n,p}}(\Gamma(C,C'))\leq \mathrm{\mod}_{F_{n,p}}(\Gamma(M_1,O)).$$

We next show that the equality case in (\ref{similarity}) cannot happen. We argue by contraction. Assume that
$$\textup{mod}_{F_{n,p}}(\Gamma(C,C'))=\textup{mod}_{F_{n,p}}(\Gamma(C,C_0)).$$
Note that all carpet modulus considered above are finite by Proposition $\ref{extremal_mass}$ and so there exist
unique extremal mass distributions, say $\rho$ and $\rho'$, for $\textup{mod}_{F_{n,p}}(\Gamma(C,C'))$ and $\textup{mod}_{F_{n,p}}(\Gamma(O,M_1))$, respectively,
by Proposition $\ref{extremal}$.

Let $\mathcal{C}$ be the set of all peripheral circles of $F_{n,p}$.
According to the description in Proposition $\ref{extremal_mass}$, $\rho$ and $\rho'$
are supported on $\mathcal{C}\setminus\{C,C'\}$ and $\mathcal{C}\setminus\{O,M_1\}$, respectively.

By transplanting $\rho'$ to the carpet $F$
using a suitable Euclidean similarity between $F$ and $F_{n,p}$,
we get an admissible mass distribution $\widetilde{\rho}$ for $F$ supported only on the set
of peripheral circles of $F$ except $C$ and $C_0$.
 Note that the total mass of $\widetilde{\rho}$ is the same as $\textup{mass}(\rho')$.

 We extend
$C\rightarrow\widetilde{\rho}(C)$ by zero if $C$ belonging to $\mathcal{C}$ does not intersect the interior region of $C_0$.
Then $\widetilde{\rho}$ is an admissible mass
distribution for $\textup{mod}_{F_{n,p}}(\Gamma(C,C_0))$, thus for $\textup{mod}_{F_{n,p}}(\Gamma(C,C'))$ as well. However, $\widetilde{\rho}\neq\rho$
and $\textup{mass}(\widetilde{\rho})=\textup{mod}_{F_{n,p}}(\Gamma(C,C'))$, we arrive at a contradiction by Proposition $\ref{extremal}$.

In summary, we get the following crucial inequality:
\begin{equation}\label{crucial_inequality}
\textup{mod}_{F_{n,p}}(\Gamma(C,C'))<\textup{mod}_{F_{n,p}}(\Gamma(O,M_1))
\end{equation}
where $\{C,C'\}\neq\{O,M_i\}\, i=1,2,3,4$ and non-adjacent. So the lemma follows.
\end{proof}

\begin{cor}\label{qs_self_maps}
Let $f$ be a quasisymmetric self-map of $F_{n,p}$. Then $$f(\{O,M_1,M_2,M_3,M_4\})=\{O,M_1,M_2,M_3,M_4\}.$$
\end{cor}
\begin{proof}
We argue by contraction. Assume that $f$ maps $\{O, M_1\}$ to some pair of peripheral circles $\{C,C'\}\nsubseteq\{O,M_1,M_2,M_3,M_4\}$
and $f(O)=C$. By Proposition
$\ref{quasiconformal_extend}$, $f$ extends to a quasiconformal homeomorphism on $\mathcal{S}^2$. In particular, $\Gamma(C,C')$=$f(\Gamma(O,M_1))$. Then Lemma \ref{lem:invariant} implies $$\textup{mod}_{F_{n,p}}(\Gamma(O,M_1))=\textup{mod}_{F_{n,p}}(\Gamma(C,C')).$$
We distinguish the argument into two cases depending on the type of the squares $C$ and $C'$, i.e., whether they are adjacent or not.

Case $(1)$: $C,C'$ are non-adjacent. This is only possible if $\{C,C'\}\subseteq\{O,M_1,M_2,M_3,M_4\}$ by Lemma $\ref{interchange}$.
Then we get a contradiction.

Case $(2)$: $C,C'$ are adjacent. Suppose $C,C'$ are inner circles of some copy $F\subset F_{n,p}$.
Consider $f(M_i)$, $i=2,3,4$. They must be inner circles of $F$ as well. Otherwise, for example,
suppose that $f(M_2)$ is not an inner circle of $F$. Since $C$ and $f(M_2)$ are non-adjacent, we can apply  Lemma \ref{interchange} to
show that
$$\textup{mod}_{F_{n,p}}(\Gamma(C,f(M_2)))<\textup{mod}_{F_{n,p}}(\Gamma(O,M_1)),$$
which is contradicted with the fact that
$$\textup{mod}_{F_{n,p}}(\Gamma(C,f(M_2)))=\textup{mod}_{F_{n,p}}(\Gamma(O,M_2))= \textup{mod}_{F_{n,p}}(\Gamma(O,M_1)).$$
As a result, $\{f(O), f(M_1),f(M_2), f(M_3),f(M_4))\}$ are pairwise adjacent and all of them are
inner circles of $F$. However, $F$ contains exactly four inner circles. So Case (2) can not happen.

By  the same argument to pairs $O$ and $M_i$, $i=2,3,4$, the corollary follows.
\end{proof}

\subsection{Quasisymmetric group $\mathrm{QS}(F_{n,p})$ is finite}

Let $H$ denote the Euclidean isometry group which consists of eight elements: four of them rotate around the center by $\pi/2,\pi,3\pi/2$,
and $2\pi$, respectively; the others are orientation-reserving and reflecting by lines $x=0$, $x=y$, $y=0$ and $x+y=0$, respectively.
It is obvious that $H$ is contained in $\mathrm{QS}(F_{n,p})$.

\begin{cor}\label{coro:finite}
Let $5\leq n, 1\leq p<\frac{n}{2}-1$ be integers. Then the group $\mathrm{QS}(F_{n,p})$ of quasisymmetric self-maps of $F_{n,p}$ is finite.
\end{cor}
\begin{proof}
According to Corollary $\ref{qs_self_maps}$, $\{O,M_1,M_2,M_3,M_4\}$ are preserved under every quasisymmetric self-map of $F_{n,p}$. The group $G$
of all orientation-preserving quasisymmetric self-maps of $F_{n,p}$ is finite by the proof of Case (1) in Corollary ($\ref{three_circle_theorem}$).
Since $G$ is a subgroup of $\textup{QS}(F_{p,q})$ with index two, $\mathrm{QS}(F_{p,q})$ is finite.
\end{proof}

\subsection{Proof of Theorem 1}

Recall that the standard carpet $S_m, m\geq 3$ odd, is obtained by subdivide $[0,1]\times[0,1]$ into $m^2$ subsquares of equal size, removing the interior of the middle square, and repeating
these operations to every subsquares, inductively.
\begin{proof}[Proof of Theorem \ref{thm:equivalent}]
Let $\mathcal{M}, \mathcal{O}$ be the inner circle and outer circle of $S_m$ respectively.
Lemma 5.1 of  \cite{BM} states that $\textup{mod}_{S_m}(\Gamma(\mathcal{O},\mathcal{M}))$ is strictly larger than the carpet modulus of any other path family $\Gamma({C},{C}')$
with respect to $S_m$, where ${C}$ and ${C}'$ are peripheral circles of $S_m$. While for carpet $F_{n,p}$, according
to the symmetry, at least two pairs of peripheral circles  the maximum of $\{\textup{mod}_{F_{n,p}}\Gamma(C_1,C_2): C_1,C_2\in\mathcal{C}\}$.
Since any quasisymmetric maps from $F_{n,p}$ to $S_m$ must preserve such a maximum property, there is no such quasisymmetric map.
\end{proof}

\section{Weak tangent spaces}\label{sec_weak_tangent_spaces}
The results in this section generalize the discussion in (\cite{BM}, Section 7).

At first, we explain the definition of weak tangent of
a carpet. Then we show that a quasisymmetric map between two carpets $F_{n,p}$ induces
a quasisymmetric map between weak tangents.

\subsection{Weak tangents} In general, the \emph{weak tangents} of a metric space $M$ at a point $p\in M$
can be defined as the Gromov-Hausdorff limits of the pointed metric spaces
$$\lim_{\lambda\to \infty} (\lambda M, p)$$
where $\lambda M$  is the same set of points with $M$ equipped with the original
metric multiplied by $\lambda$.
If the limit is unique up to multiplied by positive constants, then the weak tangents
is usually called the \emph{tangent cone} of $M$ at $p$.

In the following, as in \cite{BM}, we will use a suitable definition of weak tangents
for subsets of $\mathbb{S}^2$ equipped with the spherical metric.

Suppose that  $a,b\in\mathbb{C},a\neq0$ and $M\subseteq\widehat{\mathbb{C}}$.
 We denote by $$aM+b:=\{az+b:z\in M\}.$$ Let $A$ be a subset of $\widehat{
\mathbb{C}}$ with a distinguished point $z_0\in A$, $z_0\neq\infty$.
We say that a closed set $W_{A}(z_0)\subseteq\widehat{\mathbb{C}}$ is a
$weak\ tangent$ of $A$ if there exists a sequence $(\lambda_n)$
with $\lambda_n\rightarrow\infty$ such that the sets $A_n:=\lambda_n(A-z_0)$
converge to $W_{A}(z_0)$ as $n\rightarrow\infty$
in the sense of Hausdorff convergence on $\widehat{\mathbb{C}}$
equipped with the spherical metric.
In this case, we use the notation $$W_A(z_0)=\mathop{\textup{lim}}^{}_{n\rightarrow\infty}(A,z_0,\lambda_n).$$
Since for every sequence $(\lambda_n)$ with $\lambda_n\rightarrow\infty$,
there is a subsequence $(\lambda_{n_k})$ such that the sequence of the
sets $A_{n_k}=\lambda_{n_k}(A-z_0)$ converges as $k\rightarrow\infty$,
$A$ has weak tangents at each point $z_0\in A\setminus\{\infty\}$.
In general, weak tangents at a point are not unique.
In particular, $\lambda W_A(z_0)$ is also a weak tangent.

Now we apply the notion to our carpets $F_{n,p}$.
In fact, the following arguments work for a general class of carpets,
 such as the standard Sierpi\'nski
carpet $S_m$ and carpets which satisfy some self-similarity property.

\emph{A weak tangent of a point }$z_0\in F_{n,p}$ is a closed set
$W_{F_{n,p}}(z_0)\subseteq\widehat{\mathbb{C}}$ such that
 $$W_{F_{n,p}}(z_0)=\mathop{\textup{lim}}^{}_{j\rightarrow\infty}(F_{n,p},
z_0,n^{k_j}),$$ where $k_j\geq1$ and $k_j\rightarrow\infty$ as $j\rightarrow\infty.$

At the point $0$ the carpet $F_{n,p}$ has the unique weak tangent
\begin{equation}\label{w_F_np}
W_{F_{n,p}}(0)=\mathop{\textup{lim}}^{}_{j\rightarrow\infty}
(F_{n,p},0,n^j)=\{\infty\}\cup\bigcup_{j\in\mathbb{N}_0}n^j F_{n,p}.
\end{equation}
This follows from the inclusions $n^j F_{n,p}\subseteq
n^{j+1}F_{n,p}$.

Similarly, at each corner of $O$ there exists a unique weak tangent of
$F_{n,p}$ obtained by a suitable rotation of the set
$W_{F_{n,p}}(0)$ around $0$.

\begin{figure}[!hbp]
\centering
\includegraphics[width=0.45\linewidth]{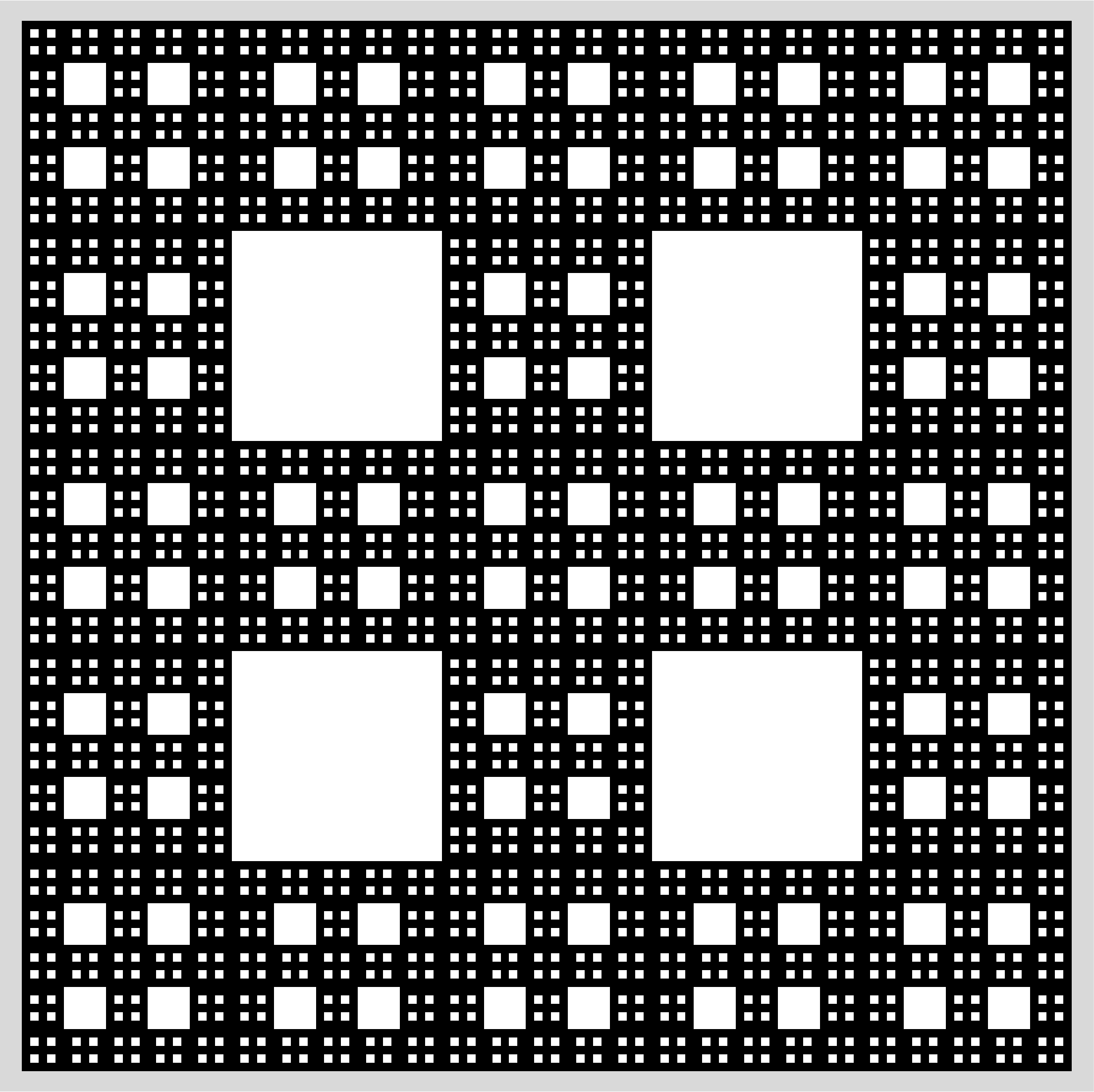}
\caption{\small {The weak tangent $W_{F_{n,p}}(0)$.}}
\label{carpet}
\end{figure}

Let $c=p/n+\mathbf{i}p/n$ be the lower-left corner of $M_1$.
Then at $c$ the carpet $F_{n,p}$ has unique weak tangent
$$W_{F_{n,p}}(c)=
\mathop{\textup{lim}}^{}_{j\rightarrow\infty}(F_{n,p},c,n^{j})=\{\infty\}\cup\bigcup_{j \in\mathbb{N}_0}n^j(\mathbf{i}F_{n,p}\cup
(-\mathbf{i})F_{n,p}\cup(-1)F_{n,p}).$$
Note that $W_{F_{n,p}}(c)$ can be obtained by pasting together three copies of $W_{F_{n,p}}$.
If $z_0$ is a corner of a peripheral circle $C\neq O$ of $F_{n,p}$,
then $F_{n,p}$ has a unique weak tangent at $z_0$ obtained
by a suitable rotation of the set $W_{F_{n,p}}(c)$ around $0$.

\begin{lem}\label{lem_weak_space_measure_zero_uniform_quasicircles}
Let $z_0$ be a corner of a peripheral circle of $F_{n,p}$.
Then the weak tangent $W_{F{n,p}}(z_0)$ is a carpet of measure zero.
If $W_{F_{n,p}}(z_0)$ is equipped with the spherical metric,
then the family of peripheral circles of $W_{F_{n,p}}(z_0)$
are uniform quasicircles and uniformly relatively separated.
\end{lem}
\begin{proof}
We can assume that $z_0$ equals $0$. The proof works for other cases.

First note that ($\ref{w_F_np}$) implies that $W_{F_{n,p}}(0)$ is a carpet of measure zero,
since $W_{F_{n,p}}(0)$ is the union of countably many sets of measure zero.

Let $\Omega=\{z\in \mathbb{C} : \mathrm{Re}(z)> 0, \mathrm{Im} (z)>0\}$.
Then $\partial \Omega$ ia a peripheral circle of $W_{F_{n,p}}(0)$.
It is easy to construct a bi-Lipschitz map between $\partial \Omega$
and the unit circle (both equipped with the spherical metric). Hence $\partial \Omega$
is a quasicircle. Note that all other peripheral circles of $W_{F_{n,p}}(0)$ are squares.
As a result, the peripheral circles of $W_{F_{n,p}}(0)$ are uniformly quasicircles.

To show that the peripheral circles are uniformly relatively separated,
we only need to check  the following inequality:
\begin{equation}\label{equ:euclidean}
\mathrm{dist}(C_1,C_2)\geq\min\{\ell(C_1), \ell(C_2)\}
\end{equation}
for any peripheral circles $C_1, C_2 \neq \partial \Omega$.
Here $\mathrm{dist}(\cdot,\cdot)$ and $\ell(\cdot)$ denote
the Euclidean distance and Euclidean side length.

The inequality implies that the peripheral circles are uniformly relatively separated
with respect to the Euclidean metric. To see that they are
uniformly relatively separated property with respect to
the spherical metric, we can apply an argument of (\cite{BM}, Lemma 7.1).

\end{proof}

\subsection{Quasisymmetric maps between weak tangents}
We are interested in quasisymmetric maps $g:W\rightarrow W'$ between weak tangents $W$ of $F_{n,p}$ and weak tangents $W'$ of
$F_{n,p}$. Note that $0,\infty\in W,W'$. We call $g$ $normalized$ if $g(0)=0$ and $g(\infty)=\infty$.
\begin{lem}\label{lem_weak_space}
Let $z_0$ be a corner of a peripheral circle of $F_{n_1,p_1}$ and let $w_0$ be a corner of a peripheral circle of $F_{n_2,p_2}$.
Suppose that $f:F_{n_1,p_1}\rightarrow F_{n_2,p_2}$ be a quasisymmetric map with $f(z_0)=w_0$. Then $f$ induces a normalized quasisymmetric map
$g$ between the weak tangent $W_{F_{n_1,p_1}}(z_0)$ and $W_{F_{n_2,p_2}}(w_0)$.
\end{lem}

\begin{proof}
By Proposition $\ref{quasiconformal_extend}$ we can extend $f$ to a quasiconformal self-homeomorphism $F$ of
$\widehat{\mathbb{C}}$. There exists a relative neighborhood $N_1$ of $z_0$ in $F_{n_1,p_1}$ and a relative neighborhood
$N_2$ of $w_0$ in $F_{n_2,p_2}$ with $F(N_1)=N_2$ such that $$W_{F_{n_1,p_1}}(0)\setminus\{\infty\}=\bigcup_{j\in\mathbb{N}_0}
n^j_1(N_1-z_0)$$
and $$W_{F_{n_2,p_2}}(0)\setminus\{\infty\}=\bigcup_{j\in\mathbb{N}_0}n^j_2(N_2-w_0)$$
Pick a point $u_0\in N-z_0$, $u_0\neq 0$. Then for each $j\in\mathbb{N}_0$ we have $F(z_0+n_1^{-j}u_0)(\neq w_0,\infty) $ in
$F_{n_2,p_2}$.

We consider the following quasiconformal self-map $F_j$ of $\widehat{\mathbb{C}}$ with $F_j(n^j_1(N_1-z_0))=n_2^{k(j)}(N_2-w_0)$:
$$F_j:u\mapsto n_2^{k(j)}(F(z_0+n_1^{-j}u)-w_0)$$ for $u\in\widehat{\mathbb{C}}$, where $k(j)$ is the unique integer
such that $1\leq|F_j(u_0)|<n_2$.

Note that $k(j)\rightarrow\infty$ as $j\rightarrow\infty$ and $F(\infty)\neq w_0$. This implies that
$F_j(\infty)\rightarrow\infty$ as $j\rightarrow\infty$. We also have $F_j(0)=0$. So the images
of $0,\infty$ and $u_0$ under $F_j$ have mutual spherical distance uniformly bounded from below independent of $j$. Moreover,
$F_j$ is obtained from $F$ by post-composing and pre-composing M\"{o}bius transformations. Hence the sequence $(F_j)$ is
uniformly quasiconformal, and it follows that we can find a subsequence of $(F_j)$ that converges uniformly on
$\widehat{\mathbb{C}}$ to a quasiconformal map $F_{\infty}$. Without loss of generality, we assume that $(F_j)$ converges uniformly
to $F_{\infty}$.

Note that $F_\infty(0)=0$ and $F_\infty(\infty)=\infty$. To prove the statement of the lemma, it suffices to show that $F_{\infty}(W_{F_{n_1,p_1}(z_0)})=W_{F_{n_2,p_2}}(w_0)$, because then
$g:=F_{\infty}|W_{F_{n_1,p_1}}(z_0)$ is an induced normalized quasisymmetric map between $W_{F_{n_1,p_1}}(z_0)$
and $W_{F_{n_2,p_2}}(w_0)$, as desired.

Let $u$ be an arbitrary point in $W_{F_{n_1,p_1}}(z_0)$. There exists a sequence $(u_j)$ with
$u_j\in n_1^j(N_1-z_0)$ converging to $u$. We have $F_j(u_j)\in n_2^j(N_2-w_0)$ and a subsequence of $(F_j(u_j))$
converging to some point $v$ in $W_{F_{n_2,p_2}}(w_0)$. By the definition of $F_{\infty}$, we have $F_{\infty}(u)=v$.
Hence $F_{\infty}(W_{F_{n_1,p_1}}(z_0))\subseteq W_{F_{n_2,p_2}}(w_0)$.

For every point $v$ in $W_{F_{n_2,p_2}}(w_0)$, there exists a sequence $(u_j)$ with $u_j\in n_1^j(N_1-z_0)$ such that
$(F_j(u_j))$ converges to $v$. Then we can choose a subsequence of $(u_j)$ converging to some point $u$ in $W_{F_{n_1,p_1}}(z_0)$
and so $F_{\infty}(u)=v$.

It follows that $F_{\infty}(W_{F_{n_1,p_1}}(z_0))=W_{F_{n_2,p_2}}(w_0)$ and we are done.
\end{proof}

We have proved in Corollary \ref{qs_self_maps} that a
quasisymmetric self-map $f$ of $F_{n,p}$ maps  $\{O,M_1,M_2,M_3,M_4\}$ to $\{O,M_1,M_2,M_3,M_4\}$.
In the remaining part of this section, we will show that there is no quasisymmetric self-map $f$ of $F_{n,p}$ with
$f(0)=c$, where $c$ is a corner of an inner circle. By Lemma \ref{lem_weak_space}, if such an $f$ exists, then it would induce a
normalized quasisymmetric  map from $W_{F_{n,p}}(0)$ to $W_{F_{n,p}}(c)$.
However, the following proposition shows that:

\begin{pro}\label{pro_weak_space_normalized_qs}
There is no normalized quasisymmetric  map from $W_{F_{n,p}}(0)$ to $W_{F_{n,p}}(c)$.
\end{pro}

To prove the proposition, we need two lemmas.

Let $G$ and $\widetilde{G}$ be the group of normalized orientation-preserving
quasisymmetric self-maps of $W_{F_{n,p}}(0)$ and $W_{F_{n,p}}(c)$, respectively.
By Corollary \ref{three_circle_theorem}, $G$ and $\widetilde{G}$ are infinite cyclic groups.
Note that the map $\mu(z):= nz$ is contained in $G\cap \widetilde{G}$.
We assume that $G=<\phi>$ and $\mu=\phi^s$ for some $s\in \mathbb{Z}_+$.
Since the peripheral circles of $W_{F_{n,p}}(0)$ are uniformly quasicircles and uniformly relatively
separated, there exists a quasiconformal extension $\Phi: \widehat{\mathbb{C}}\to \widehat{\mathbb{C}}$ of $\phi$.
Let $H$ be the group generated by the reflection in the real and in the imaginary axes. We may assume that $\Phi$ is equivalent under the
 action of $H$ (see Page 42, \cite{BM} for the discussion).

Let $\Omega=\{z\in \mathbb{C} : \mathrm{Re}(z)> 0, \mathrm{Im} (z)>0\}$.
Then $C_0:=\partial \Omega$ is a peripheral circle of $W_{F_{n,p}}(0)$.
Since $\Phi(C_0)=C_0$ and $\Phi$ is orientation-preserving, $\Phi(\Omega)=\Omega$.

Let $\Gamma$ be the family of all open paths in $\Omega$ that connects the positive real
and positive imaginary axes. Since the paths in $\Omega$ are open, they don't intersect
with $C_0$. For any peripheral circle $C$ of $W_{F_{n,p}}(0)$ that meets some path in $\Gamma$,
note that $\phi^k(C)\neq C$ for all $k\in \mathbb{Z}\setminus \{0\}$ (otherwise, $\phi$
would be of finite order, contradicted with the fact that $\phi$ is the generator of the infinite cyclic group $G$).
So we can apply Lemma \ref{cyclic} to conclude that

$$\mathrm{mod}_{W_{F_{n,p}}(0)/<\mu>}(\Gamma)=\mathrm{mod}_{W_{F_{n,p}}(0)/<\phi^s>}(\Gamma)=s\mathrm{mod}_{W_{F_{n,p}}(0)/G}(\Gamma).$$

Note that without the action of the group $G$, the carpet modulus
$\mathrm{mod}_{W_{F_{n,p}}(0)}(\Gamma)$ is equal to infinity.

\begin{lem}\label{lemma1}
We have $0< \mathrm{mod}_{W_{F_{n,p}}(0)/G}(\Gamma)<\infty$.
\end{lem}
\begin{proof}
Let us first show that $\mathrm{mod}_{W_{F_{n,p}}(0)/<\mu>}(\Gamma)<\infty$ by constructing an admissible
mass distribution of finite mass.

Let $pr: \mathbb{C}\setminus\{0\} \to \mathbb{S}^1$ be the projection $z\mapsto \frac{z}{|z|}$.
If $C\neq C_0$ is a peripheral circle of $W_{F_{n,p}}(0)$, we let $\theta(C)$ be the arc length
of $pr(C)$. We set

\begin{equation*}
\rho(C):=
\begin{cases}
0, \mathrm{if} \  C=C_0; \\
\frac{2}{\pi}\theta(C), \mathrm{if} \  C\neq C_0.
\end{cases}
\end{equation*}
Note that $\rho$ is $<\mu>$-invariant.

Let $\Gamma_0$ be the family of paths $\gamma\in \Gamma$ that are not locally rectifiable or
for which $\gamma\cap W_{F_{n,p}}(0)$ has positive length. Since $W_{F_{n,p}}(0)$ is a set of measure
zero, we have $\mathrm{mod}(\Gamma_0)=0$, i.e., $\Gamma_0$ is an exceptional subfamily of $\Gamma$.

For any $\gamma\in \Gamma\setminus\Gamma_0$, note that
$$\sum_{\gamma\cap\C \neq \emptyset} \rho(C)=\frac{2}{\pi}\sum_{\gamma\cap\C\neq \emptyset} \theta(C)\geq 1.$$
As a result, $\rho$ is admissible.

Let $Q_0=[0,1]\times[0,1]$. Note that every $<\mu>$-orbit of a peripheral circle $C\neq C_0$
has a unique element contained in the set $F=\overline{\mu(Q_0)\setminus Q_0}$.
There is a constant $K>0$ such that
$$\theta(C)\leq K\ell(C)$$
for all peripheral circles $C\subset F$.
It follows that
$$\frac{4}{\pi^2}\sum_{C\subset F}\theta(C)^2 \lesssim \sum_{C\subset F}\ell(C)^2=Area(F)=n^2-1.$$
Hence $\rho$ is a finite admissible mass distribution for $\mathrm{mod}_{W_{F_{n,p}}(0)/<\mu>}(\Gamma)$.

To show that $ \mathrm{mod}_{W_{F_{n,p}}(0)/<\mu>}(\Gamma)>0$, we only need to show that
the carpet satisfies the assumptions in Proposition \ref{extremal-group}.
Then the extremal mass distribution for $\mathrm{mod}_{W_{F_{n,p}}(0)/<\mu>}(\Gamma)$ exists and this
is only possible if $\Gamma$ itself is an exceptional family, that is, $\mathrm{mod}(\Gamma)=0$.

In fact, for $k\in \mathbb{N}$ we let $\mathcal{C}_k$ be the set of all peripheral circles $C$ of
$W_{F_{n,p}}(0)$ with $C\subset F_k= \overline{\mu^{k}(Q_0)\setminus \mu^{-k}(Q_0)}$.
Then
\begin{enumerate}
\item Every $<\mu>$-orbit of a peripheral circle $C\neq C_0$ has exactly $2k$ elements in $\mathcal{C}_k$.

\item Let $\Gamma_k$ be the family of paths in $\Gamma$ that only meet peripheral circles in $\mathcal{C}_k$.
Then $\Gamma=\bigcup_k \Gamma_k$.
\end{enumerate}
As a result, the assumptions in Proposition \ref{extremal-group} are satisfied.
\end{proof}

Let  $\widetilde{\Omega}=\mathbb{C} \setminus \overline{\Omega}$. The closure of $\widetilde{\Omega}$
contains $W_{F_{n,p}}(c)$ and $C_0=\partial \Omega=\partial \widetilde{\Omega}$
is a peripheral circle of $W_{F_{n,p}}(c)$. Denote $\psi=\Phi |_{W_{F_{n,p}}(c)}$. Then we have $\psi\in \widetilde{G}$.
Let $\widetilde{\Gamma}$ be the family of all open paths in $\widetilde{\Omega}$ that join the positive real
and the positive imaginary axes.

\begin{lem}\label{lemma2}
We have $\mathrm{mod}_{W_{F_{n,p}}(c)/<\psi>}(\widetilde{\Gamma})\leq \frac{1}{3}\mathrm{mod}_{W_{F_{n,p}}(0)/G}(\Gamma)$.
\end{lem}
\begin{proof}
Let $\rho$ be an arbitrary admissible invariant mass distribution for $\mathrm{mod}_{W_{F_{n,p}}(0)/G}(\Gamma)$,
with exceptional family $\Gamma_0\subset \Gamma$. We set

\begin{equation*}
\widetilde{\rho}(\widetilde{C}):=
\begin{cases}
0, \mathrm{if} \  \widetilde{C}=C_0; \\
\frac{1}{3}\rho(\alpha(\widetilde{C}))
\end{cases}
\end{equation*}
if there is an $\alpha\in H$ such that $\alpha(\widetilde{C})$ is a peripheral circle of $W_{F_{n,p}}(0)$
(such an $\alpha$ exits and is unique).

Since $\Phi$ is $H$-equivalent and $\rho$ is $G$-invariant, $\widetilde{\rho}$ is $<\psi>$-invariant.

Let $\widetilde{\Gamma}_0$ be the family of paths in $\widetilde{\Gamma}$ that have a subpath that can be
mapped to a path in $\Gamma_0$ by an element of $\alpha\in H$. Then $\mathrm{mod}(\widetilde{\Gamma}_0)=0$.

Let $\gamma\in \widetilde{\Gamma}$. Note that $\gamma$ has three disjoint open subpaths:
one for each quarter-plane of $\widetilde{\Omega}$ and by suitable elements in $H$, the three subpaths are mapped to
paths in $\Gamma$. Denote the images by  $\gamma_1, \gamma_2,
\gamma_3$.
If $\gamma\in \widetilde{\Gamma}\setminus\widetilde{\Gamma}_0$, then $\gamma_i\in \Gamma\setminus\Gamma_0, i=1,2,3$
and
$$\sum_{\gamma\cap \widetilde{C}\neq \emptyset}\widetilde{\rho}(\widetilde{C})
\geq \frac{1}{3}\sum_{i=1}^3 \sum_{\gamma_i\cap C\neq \emptyset}\rho(C)\geq 1.$$
Hence $\widetilde{\rho}$ is admissible for $\mathrm{mod}_{W_{F_{n,p}}(c)/<\psi>}(\widetilde{\Gamma})$
and
$$\mathrm{mod}_{W_{F_{n,p}}(c)/<\psi>}(\widetilde{\Gamma})\leq \mathrm{mass}_{W_{F_{n,p}}(c)/<\psi>}(\widetilde{\rho})
\leq \frac{1}{3}\mathrm{mass}_{W_{F_{n,p}}(0)/G}(\rho).$$
Since $\rho$ is an arbitrary
 mass distribution for $\frac{1}{3}\mathrm{mod}_{W_{F_{n,p}}(0)/G}(\Gamma)$,
 the statement follows.

 \end{proof}

\begin{proof}[Proof of Proposition \ref{pro_weak_space_normalized_qs}]
Suppose not, there exists a normalized quasisymmetric map
$f: W_{F_{n,p}}(0) \to W_{F_{n,p}}(c)$. Precomposing $f$ by the reflection in the diagonal line $\{x=y\}$
if necessary, we may assume that $f$ is orientation-preserving.
Then $\widetilde{G}=f\circ G\circ f^{-1}$ and
$\widetilde{\phi}=f\circ \phi\circ f^{-1}$ is a generator for $\widetilde{G}$.

Let $F:\widehat{\mathbb{C}}\to \widehat{\mathbb{C}}$ be a quasiconfomral extension of $f$.
Then $\widetilde{\Gamma}=F(\Gamma)$. By quasisymmetric invariance of carpets modulus,
$$\mathrm{mod}_{W_{F_{n,p}}(c)/\widetilde{G}}(\widetilde{\Gamma})=\mathrm{mod}_{W_{F_{n,p}}(0)/G}(\Gamma).$$

Assume that $\psi=\widetilde{\phi}^m$. Then similar to our discussion before Lemma \ref{lemma1}, we have

$$\mathrm{mod}_{W_{F_{n,p}}(c)/<\psi>}(\widetilde{\Gamma})=|m|\mathrm{mod}_{W_{F_{n,p}}(c)/\widetilde{G}}(\widetilde{\Gamma}).$$

Hence by Lemma \ref{lemma2} we have
\begin{eqnarray*}
\mathrm{mod}_{W_{F_{n,p}}(0)/G}(\Gamma)&=& \mathrm{mod}_{W_{F_{n,p}}(c)/\widetilde{G}}(\widetilde{\Gamma}) \\
&=& \frac{1}{|m|}\mathrm{mod}_{W_{F_{n,p}}(c)/<\psi>}(\widetilde{\Gamma}) \\
&\leq& \frac{1}{3|m|} \mathrm{mod}_{W_{F_{n,p}}(0)/G}(\Gamma).
\end{eqnarray*}
This is possible only if $\mathrm{mod}_{W_{F_{n,p}}(0)/G}(\Gamma)$ is equal to $0$ or $\infty$.
But this is contradicted with Lemma \ref{lemma1}.

\end{proof}

\section{Quasisymmetric rigidity}\label{sec_proof_of_theorems}
Let $D$ be  the diagonal $\{(x,y)\in \mathbb{R}^2 : x=y\}$ and
$V$ be the vertical line $\{(x,y)\in \mathbb{R}^2 : x=\frac{1}{2}\}$.
We denote the reflections in $D$ and $V$ by $R_D$ and $R_V$, respectively.
The Euclidean isometry group  of $F_{n,p}$ is generated by $R_D$ and $R_V$.

Let $\mathrm{QS}(F_{n,p})$ be the group of quasisymmetric self-maps of $F_{n,p}$.
By Corollary \ref{coro:finite}, $\mathrm{QS}(F_{n,p})$ is a finite group.

\begin{pro}\label{pro_noninterchange}
 Let  $f$ be a quasisymmetric self-map of $F_{n,p}$.
Then $f(\{O\})=\{O\}$ and $f(\{M_1,M_2,M_3,M_4\})=\{M_1,M_2,M_3,M_4\}$.
\end{pro}
\begin{proof}
From Corollary $\ref{qs_self_maps}$, we argue by contraction and assume that there exists a quasisymmetric self-map $f$ of $F_{n,p}$ and some $i\in\{1,2,3,4\}$ such
that $f(\{O\})=\{M_i\}$. By pre-composing and post-composing suitable elements in the Euclidean isometry group, we can
suppose that $f$ is orientation-preserving and $f(\{O\})=\{M_1\}$.

Let $G$ be the subgroup of $\mathrm{QS}(F_{n,p})$,
$$G=\{g\in\textup{QS}(F_{n,p})\ | \ g({O})={O},g({M_1})={M_1}\}.$$
$G$ has a subgroup $G'$ with index two  consisting of
orientation-preserving elements. Then
$$G=G'\bigsqcup G'\circ R_D.$$
We denote by
$$\mathcal{O}_{G}(z)=\{g(z):g\in G\}$$ the orbit of $z$ under the action of $G$ for arbitrary $z\in F_{n,p}$ .
Let $c=(p/n,p/n)$ and $c'=((p+1)/n,(p+1)/n)$ be the lower-left and upper-right corners of  $M_1$, respectively.

Now we consider the map
\begin{eqnarray*}
\Phi_0&:& G'\longrightarrow\mathcal{O}_{G}(0)\\
&& g\longmapsto g(0).
\end{eqnarray*}
Note that $\Phi_0$ is an isomorphism. In fact,
for any $g(0)\in \mathcal{O}_{G}(0)$, if $g$ is orientation-preserving, then $\Phi_0(g)=g(0)$; otherwise, $\Phi_0(g\circ R_D)=g(0)$.
So $\Phi_0$ is a surjection. On the other hand, if $\Phi_0(g_1)=\Phi_0(g_2)$for any $g_1, g_2\in G'$, then Case (2) of Corollary
$\ref{three_circle_theorem}$ gives $g_1=g_2$. So $\Phi_0$ is a injection.

Similarly, we can also define the isomorphism
\begin{eqnarray*}
\Phi_c &:& G'\longrightarrow\mathcal{O}_{G}(c)\\
&& g\longmapsto g(c).
\end{eqnarray*}
These isomorphisms $\Phi_0$ and $\Phi_c$ imply that
\begin{equation}
\#\mathcal{O}_{G}(0)=\#G'=\#\mathcal{O}_{G}(c)
\end{equation}
On the other hand, $f$ induces the following isomorphism
\begin{eqnarray*}
f_*&:& G\longrightarrow G \\
&& g\longmapsto f\circ g\circ f^{-1}.
\end{eqnarray*}
We denote by $m=f(0)$. Then
\begin{eqnarray*}
\mathcal{O}_G(m)&=&\{g(m):g\in G\}=\{f\circ g\circ f^{-1}(m):g\in G\} \\
&=&\{f\circ g(0):g\in G\}=f(\mathcal{O}_G(0)).
\end{eqnarray*}
Hence
\begin{equation*}
\#\mathcal{O}_G(m)=\#G'=\#\mathcal{O}_G(0)
\end{equation*}
and so the orbits $\mathcal{O}_G(m)$ and $\mathcal{O}_G(c)$ have the same number of elements.

If $G'\neq\{\mathrm{id}\}$, we claim that $G'$ is a cyclic group of order $3$.
Indeed, for any $g\neq \mathrm{id}$ in $G'$, $g(M_3)\neq M_3$, otherwise Case $(1)$ of Corollary $\ref{three_circle_theorem}$ implies $g=\textup{id}$.
By Corollary $\ref{qs_self_maps}$, either $g(M_3)=M_4,g(M_4)=(M_2)$ or $g(M_3)=M_2, g(M_2)=(M_4)$.
In both cases, $g$ is of order $3$, a.e., $g^3=\mathrm{id}$. Use Corollary $\ref{qs_self_maps}$
again we know that $G'$ is generated by $g$. So the claim follows.

Hence, we have $\#\mathcal{O}_G(m)=\#G'=1\textup{ or }3$.
There must be some $h\in G$ with $h(m)=c\textup{ or }c'$.
Otherwise, $\mathcal{O}_G(m)$ does not
contain $c, c'$. For any point $p\in \mathcal{O}_G(m)$, the point $R_D(p)\in \mathcal{O}_G(m)$ and $R_D(p)\neq p$.
Then $\#\mathcal{O}_G(m)$ is even, which is impossible.

By Lemma $\ref{lem_weak_space}$, $h\circ f$ induces a normalizaed quasisymmetric map between the weak tangent $W_{F_{n,p}}(0)$ and $W_{F_{n,p}}(c)$
or $W_{F_{n,p}}(c')$. This contradicts Proposition $\ref{pro_weak_space_normalized_qs}$. So we have proved the proposition.
\end{proof}

\begin{proof}[Proof of Theorem 2]
We adopt the notations as in previous. The proof of Proposition $\ref{pro_noninterchange}$ implies that $G'$ is a cyclic group of order $3$
or a trivial group. To prove the theorem, it suffices to show that the former case cannot happen. We argue by contraction and assume that $G'$ is a cyclic
group of order $3$.

By Theorem $\ref{uniformization_by_round_carpets}$, there exists a quasisymmetric map $f$ from $F_{n,p}$ onto some round carpet $S$. After post-composing
suitable fraction linear transformation, we can assume that the $f(O)$ is the unit disc $\mathbb{D}$ and $f(M_1)$ lies in $\mathbb{D}$ with center $(0,0)$.
Then $f$ induces the isomorphism
\begin{eqnarray*}
f_* &:& QS(F_{n,p})\longrightarrow QS(S)\\
&& g\longmapsto f\circ g\circ f^{-1}.
\end{eqnarray*}
Combined with Theorem $\ref{quasisymmetric_rigidity_of_round_carpets}$,  $f_*(G')$ is a cyclic group
of order $3$ consisting of M\"{o}bius transformations. Moreover, elements in $f_*(G')$ preserves $\partial\mathbb{D}$ and the circle $O_1=f(M_1)$. Hence we have
$$f_*(G')=\{\textup{id},z\mapsto e^{2\pi i/3}z,z\mapsto\ e^{4\pi i/3}z\}.$$

\textbf{Claim:} $O_2=f(M_2),O_3=f(M_3),O_4=f(M_4)$
are round circles with the same diameter and  equidistributed clockwise
in the annuli bounded by $\partial\mathbb{D}\textup{ and }O_1$.

\emph{Proof of the claim}: In fact, by the proof of Proposition \ref{pro_noninterchange}, we may assume that $G'=<g>$,
where $g(M_3)=M_4, g(M_4)=M_2$ and $g(M_2)=M_3$. Note that
\begin{eqnarray*}
O_3&=& f(M_3)=f\circ g (M_2)) \\
&=& f\circ g\circ f^{-1} (O_2)
\end{eqnarray*}
where $f\circ g\circ f^{-1}$ is equal to the rotation $z\mapsto e^{2\pi i/3}z$.
Similarity, one can show that $O_4=f\circ g\circ f^{-1} (O_3)$. As a result, $O_3$ is obtained from
$O_2$ by a rotation of angle $2\pi/3$ and $O_4$ is obtained from
$O_2$ by a rotation of angle $4\pi/3$. The claim follows.

Let $R$ be the rotation in the isometry group of $F_{n,p}$ with $R(M_1)=M_2, R(M_2)=M_3, R(M_3)=M_4$, and $R(M_4)=M_1$. By Theorem $\ref{quasisymmetric_rigidity_of_round_carpets}$,
the composition
$$h=f\circ R\circ f^{-1}:S\rightarrow S$$ is also a M\"{o}bius transformation which maps $\partial\mathbb{D}
\to\partial\mathbb{D},
O_2\to O_3, O_3\mapsto O_4$.
Such a M\"{o}bius transformation must be $\varphi=z\to e^{2\pi i/3}z$. If not, let $\varphi'$ be  other M\"{o}bius
transformation satisfy the conditions. Then $\varphi'\circ\varphi^{-1}$ fixes three non-concentric circles $\partial\mathbb{D}, O_2
$ and $O_3$ and so $\varphi'\circ\varphi^{-1}=id$. Hence $\varphi'=\varphi$.
But $h(O_1)=O_2$, which is impossible. So the theorem follows.
\end{proof}

\begin{proof}[Proof of Theorem \ref{thm:equivalent}]
Suppose there exists a quasisymmetric map $f:F_{n,p}\rightarrow F_{n',p'}$.

Firstly, we claim that $f(O)=O'$, $f(\{M_1,M_2,M_3,M_4\})=\{M'_1,M'_2,M'_3,M'_4\}$.
Indeed, from Theorem 2, we know that every quasisymmetric self-map of $F_{n,p}$ and $F_{n',p'}$ is isometry and so
preserves the peripheral circle $O$ and $O'$. For any $g$ in $QS(F_{n,p})$, $f\circ g\circ f^{-1}$ is a quasisymmetric self-map of $F_{n',p'}$ and
$f\circ g\circ f^{-1}(f(O))=f(O)$. So $f(O)$ is fixed by any element in $QS(F_{n',p'}$. Hence we have $f(O)=O'$. If for some inner circles $M_i$, say $M_1$, of
 $F_{n,p}$, $f(M_1)$ is not an inner circle of $F_{n',p'}$, then by Proposition
$\ref{quasiconformal_extend}$, $f$ extension to a quasiconformal self-map of $\mathbb{S}^2$.
We have
$$\textup{mod}_{F_{n,p}}(\Gamma(M_1,O))=\textup{mod}_{F_{n',p'}}(\Gamma(f(M_1),O'))$$
and
$$\textup{mod}_{F_{n',p'}}(\Gamma(M'_1,O'))=\textup{mod}_{F_{n,p}}\Gamma(f^{-1}(M'_1),O).$$
While Lemma $\ref{interchange}$ implies $$\textup{mod}_{F_{n',p'}}(\Gamma(f(M_1),O))<\textup{mod}_{F_{n',p'}}(\Gamma(M'_1,O))$$
and
$$\textup{mod}_{F_{n,p}}\Gamma(f^{-1}(M'_1),O)\leq\textup{mod}_{F_{n,p}}(\Gamma(M_1,O)).$$
Hence $\textup{mod}_{F_{n,p}}(\Gamma(M_1,O))<\textup{mod}_{F_{n,p}}(\Gamma(M_1,O))$ and we get a contraction.

Secondly, by pre-composing and post-composing with Euclidean isometries, we
can assume that $f$ is orientation-preserving and $f(M_1)=M'_1$. We claim
that $f((0,0))=(0,0)$ and $f((1,1))=(1,1)$ or interchanges them and $f(M_3)=M'_3$. In fact, the orientation-preserving quasisymmetric map $$f^{-1}\circ R_D\circ f\circ R_D:F_{n,p}\rightarrow F_{n,p}$$
fixes peripheral circles $O$ and $M_1$. Then, by Theorem \ref{thm:rigidity}, $f^{-1}\circ R_D\circ f\circ R_D$ is a Euclidean isometry and
so it is the identity on $F_{n,p}$. This implies $f\circ R_D=R_D\circ f$. Hence the claim follows.

We now distinguish two cases to analyze.

Case $(1)$ $f((0,0))=(0,0)$ and $f((1,1))=(1,1)$.

We denote the reflection in the line $\{(x,y)\in\mathbb{R}^2:x+y=1\}$ by $R_D'$. Then the map $f^{-1}\circ R_D'\circ f \circ R_D'$ is an orientation-preserving
quasisymmetric map in $QS(F_{n,p})$, fixes peripheral circles $O,M_1$, and the point $(0,0)$. Hence this map is the identity on $F_{n,p}$ and so
$f\circ R_D'=R_D'\circ f$. It follows that $f$ fixes $(1,0)$ and $(0,1)$ or interchanges them. Since $f$ is orientation-preserving, the latter cannot happen.
By Theorem $\ref{thm_square_rigidity}$ the map $f$ must be the identity. Hence $(n,p)=(n',p')$.

Case $(2)$ $f((0,0))=(1,1)$ and $f((1,1))=(0,0)$.

The map $g=R_D\circ f\circ R_D':F_{n,p}\rightarrow F_{n',p'}$ is an orientation-preserving quasisymmetry which fixes points $(0,0)$ and $(1,1)$
and peripheral circle $O$ and maps $M_1$ to $M'_3$. Similar to Case $(1)$, $g^{-1}\circ R_D'\circ g\circ R_D'$ is an orientation-preserving isometry map fixing $(0,0),(1,1)$ and $O$ and so is the identity.
Then $g$ fixes $(1,0)$ and $(0,1)$ or interchanges them. The orientation-preserving of $g$ implies the latter case is impossible. By Theorem
$\ref{thm_square_rigidity}$ the map $g$ is the identity, which contradicts with $g(M_1)=M'_3$. So case $(2)$ can not happen.

\end{proof}



\end{document}